%% file: Cufaro_Sabino_CTSOUCTS.tex
\documentclass[showpacs,preprintnumbers,amsmath,amssymb,12pt]{article}
\usepackage{a4wide,amsmath,amsthm,epsfig,graphicx}
\usepackage{amsmath,amsthm,amssymb}
\usepackage{amsfonts}
\usepackage{graphics}
\usepackage{dsfont}
\usepackage{bbm}
\usepackage{bm}
\usepackage{xcolor}
\usepackage{dcolumn}
\usepackage{fancyhdr,fancybox}
\usepackage[svgnames]{pstricks}
\usepackage{pst-text}
\usepackage{pst-plot,pst-eucl,pstricks-add}
\usepackage{caption}
\usepackage{subcaption}
\usepackage{systeme, mathtools}
\usepackage{algorithm}
\usepackage{algpseudocode}

\algrenewcommand\alglinenumber[1]{{\sffamily\footnotesize#1}}
\algnewcommand\Not{\textbf{not}}
\algnewcommand\FALSE{FALSE}
\algnewcommand\TRUE{TRUE}

\newtheorem{thm}{Theorem}[section]

\newtheorem{prop}[thm]{Proposition}

\newtheorem{remark}{Remark}

\newcommand{\sde}{\emph{SDE}}

\newcommand{\refeqq}[1]{~(\ref{#1})}
\newcommand{\myref}[1]{~\ref{#1}}
\newcommand{\mycite}[1]{~\cite{#1}}

\newcommand{\rv}{\textit{rv}}
\newcommand{\id}{\textit{id}}
\newcommand{\iid}{\textit{iid}}
\newcommand{\sd}{\textit{sd}}

\newcommand{\LKh}{L\'{e}vy-Khintchin}

\newcommand{\pdf}{\textit{pdf}}

\newcommand{\chf}{\textit{chf}}

\newcommand{\Levy}{L\'{e}vy}

\newcommand{\Pqo}{\bm{P}\hbox{-\emph{a.s.}}}
\newcommand{\eqd}{\stackrel{d}{=}}

\newcommand{\PR}[1]{\bm{P}\left\{{#1}\right\}}
\newcommand{\EXP}[1]{\bm{E}\left[{#1}\right]}

\newcommand{\lch}{\textit{lch}}

\newcommand{\poiss}{\mathcal{P}}

\newcommand{\unif}{\mathcal{U}}

\newcommand{\gam}{\mathcal{G}}

\newcommand{\arem}{$a$-remainder}

\newcommand{\cts}{\mathcal{CTS}}

\DeclareMathOperator{\sign}{sign}

\usepackage[a4 paper, top = 1.3 cm, bottom = 1.5 cm, left = 1.3 cm, right = 1.3 cm]{geometry}
\title{\Huge \textbf{Tempered stable distributions\\ and finite variation\\ Ornstein-Uhlenbeck
processes\footnote{The views, opinions, positions or strategies expressed in this article are those of the authors
and do not necessarily represent the views, opinions, positions or strategies of, and should not be
attributed to E.ON SE.}}}

\author{Nicola \textsc{Cufaro Petroni}\footnote{cufaro@ba.infn.it}  \\
Dipartimento di \textsl{Matematica} and \textsl{TIRES}, Universit\`a di Bari\\
\textsl{INFN} Sezione di Bari\\ \vspace{7pt}
via E. Orabona 4, 70125 Bari, Italy\\
Piergiacomo
\textsc{Sabino}\footnote{piergiacomo.sabino@eon.com}\\
Quantitative Modelling \\
E.ON SE\\
\vspace{5pt}
 Br\"{u}sseler Platz 1, 45131 Essen, Germany
}

\date{}
\begin{document}
    \maketitle \thispagestyle{empty}
        \begin{abstract}
Constructing \Levy-driven Ornstein-Uhlenbeck processes is a task
closely related to the notion of self-decomposability. In
particular, their transition laws are linked to the properties of
what will be hereafter called the \emph{\arem}\ of their
self-decomposable stationary laws. In the present study we fully
characterize the \Levy\ triplet of these \arem s and we provide a
general framework to deduce the transition laws of the finite
variation Ornstein-Uhlenbeck processes associated with tempered
stable distributions.  We focus finally on the subclass of the
exponentially-modulated tempered stable laws and we derive the
algorithms for an exact generation of the skeleton of
Ornstein-Uhlenbeck processes related to such distributions, with the
further advantage of adopting a procedure computationally more
efficient than those already available in the existing literature.

        \end{abstract}
        \section{Introduction}\label{sec:introduction}

The \Levy-driven Ornstein-Uhlenbeck (OU) processes have attracted
considerable interest in recent studies because of their potential
applications to a wide range of fields. As observed in
Barndorff-Nielsen and Shephard\mycite{BNSh01}, these OU process are
mathematically tractable and can be seen as the continuous-time
analogues of the autoregressive AR($1$) processes (see
Wolfe\mycite{Wolfe1982}). They constitute indeed a rich and flexible
class that can accommodate features such as jumps, semi-heavy tails
and asymmetry which are well evident in the real physical phenomena
as well as in the financial data. The energy and commodity markets
exhibit for instance a strong mean-reversion and sudden spikes which
makes the use of the \Levy-driven OU-processes more advisable than
the standard Gaussian framework. In addition, several approaches
based on \Levy\ processes such as the Variance Gamma (VG) or the
Normal Inverse Gaussian (NIG) have been proposed to overcome the
known limits of the usual Black-Scholes model (see Madan and
Seneta\mycite{MadanSeneta90} and Barndorff-Nielsen\mycite{BN98}):
all these non Gaussian noises can of course be adopted as the
drivers of OU processes. Example of their applications to
mathematical finance can be found in Benth and
Pircalabu\mycite{BenthPircalabu18} and Cufaro Petroni and
Sabino\mycite{cs20_2} in the context of energy markets, in Bianchi
and Fabozzi\mycite{Bianchi2015} for the modeling of credit risk and
in Barndorff-Nielsen and Shephard\mycite{BNSh01} for stochastic
volatility modeling.

The distributional properties of a non-Gaussian process of OU-type
are closely related to the notion of self-decomposability (\sd),
because as noted in Barndorff-Nielsen and Shephard\mycite{BNSh01}
and in Taufer and Leonenko\mycite{TAUFER2009}, the stationary law of
such a process must be \sd. We recall here that a law with \chf\
$\eta(u)$ is said to be \sd\ (see Sato\mycite{Sato}, Cufaro
Petroni~\cite{Cufaro08}) when for every $0<a<1$ we can find another
law with \chf\ $\chi_a(u)$ such that
\begin{equation}\label{aremchf}
    \eta(u)=\eta(au)\chi_a(u).
\end{equation}
Of course a random variable (\rv) $X$ with \chf\ $\eta(u)$ is also
said to be \sd\ when its law is \sd, and looking at the definitions
this means that for every $0<a<1$ we can always find two
\emph{independent} \rv's -- a $Y$ with the same law of $X$, and a
$Z_a$ with \chf\ $\chi_a(u)$ -- such that in distribution
\begin{equation}\label{sdec-rv}
    X\eqd aY+Z_a
\end{equation}
Hereafter the \rv\ $Z_a$ will be called the \emph{\arem} of $X$ and
in general has an infinitely divisible distribution (\id) (see
Sato\mycite{Sato}). We will show in the following (see also
Barndorff-Nielsen~\cite{BN98}, Sabino and Cufaro Petroni\mycite{cs20_1} and Sabino\mycite{Sabino20b}) that the
transition law between the times $t$ and $t+\Delta t$ of a
\Levy-driven OU process $X(t)$ essentially coincides indeed with the
law of the \arem\ of its \sd-stationary distribution, provided that
$a=e^{-b\Delta t}$ where $b$ is the OU mean-reversion rate. It is
therefore natural to investigate the properties of the \arem\ of a
certain \sd\ law even irrespective of its possible relation to the
theory of the OU processes.

The first step in this inquiry apparently is the characterization of
the \Levy\ triplet of the \arem\ of a \sd\ law: this would of course
constitute a crucial building block in the construction of a
\Levy-driven OU processes. We will thereafter focus our attention on
the class of the tempered stable (TS) distributions (see for
instance see Rosi\'nski\mycite{ROSINSKI2007677} and
Grabchak\mycite{Grabchak16}) with finite variation, and we will
provide a general framework to derive their transition laws from
their associated \arem s.  There are in fact two standard ways to
associate a TS distribution to an OU process $X(\cdot)$:  if its
stationary law is a TS distribution we will say that $X(\cdot)$ is a
TS-OU process; if on the other hand, $X(\cdot)$ is driven by a TS
background noise we will say that $X(t)$ is a OU-TS process.

Some of the results discussed in the following sections are not new:
Kawai and Masuda\mycite{KawaiMasuda_1, KawaiMasuda_2} and
Zhang\mycite{Zhang11}, for instance, have considered OU processes
whose stationary marginal law is an exponentially-modulated TS
distribution hereafter called a classical TS (CTS), whereas Bianchi
et al.\mycite{BRF17} have taken into account the rapidly decreasing
TS laws (RDTS), and Grabchak\mycite{Grabchak20} finally merged all
these laws in the larger class of the general TS distributions.
Recently Qu et al\mycite{QDZ20} have also studied both CTS-OU and
OU-CTS processes. In this perspective a first contribution of the
present paper consists in harmonizing these results in the scheme of
the \arem s, and in showing the further advantages of this approach:
some of the theorems in the aforementioned literature become indeed
special cases of this proposed comprehensive framework. On the other
hand -- by exploiting properties valid for every \id\ distribution
-- an explicit knowledge of the \Levy\ triplet of the \arem\ makes
very easy and straightforward the calculation of the cumulants of
the transition law of a OU process. This turns out in particular to
be a remarkable asset in testing the efficiency of the simulation
algorithms and can be adopted for the parameters estimation. It is
also worthwhile remarking that the laws of the \arem s are in fact
\id, and therefore they lend the possibility of producing an entire
new class of associated \Levy\ processes as shown for instance in
Gardini et al.\mycite{Gardini20a, Gardini20b}.

Finally, as done in Qu et al\mycite{QDZ20}, we focus our attention
on the case of the CTS related OU processes with their transition
laws, both for the OU-CTS and the CTS-OU cases, but within the
perspective of the \arem s. We also derive a few new algorithms
intended to simulate the skeleton of such processes. We find in
particular that for the simulation of the CTS-OU processes our
procedure is computationally more efficient than that of
Zhang\mycite{Zhang11} because it does not rely on an acceptance
rejection method (other than that required to draw from a CTS law),
but rather on the inverse method (see Devroye\mycite{Dev86}). We
adopted instead a procedure based on an acceptance rejection method
for the OU-CTS process that however, at variance with that of Qu et
al\mycite{QDZ20}, has the advantage of having an expected number of
iterations before acceptance that can always be kept arbitrarily
close to $1$.

The paper is organized as follows: the
Section\myref{sec:Preliminaries} introduces the notations and the
preliminary notions. In particular it presents the basic properties
of non-Gaussian OU processes and their relation with the \sd\ laws.
In the Section\myref{sec:tsou} we then derive the \Levy\ triplet of
the \arem\ of an arbitrary \sd\ law, and in particular that of the
\arem\ of a general TS distribution with finite variation. These
results are instrumental to explicitly write down the transition law
of CTS-OU process. In the subsequent Section\myref{sec:outs} we
focus on the OU-CTS processes, and in Section\myref{sec:sim} we
present the algorithms for the simulation of the skeleton of both
the CTS-OU and the OU-CTS processes pointing out their differences
and their advantages with respect to the solutions already existing
in the literature. The Section\myref{sec:num:exp} illustrates the
effectiveness of our simulation schemes by comparing the true values
and the Monte Carlo estimated values of the first four cumulants. We
also consider some approximation schemes to further check the
performance of our procedures and we propose a simple approach to
the parameters calibration. Finally the
Section\myref{sec:conclusions} concludes the paper with an overview
of future inquiries and of possible further applications.

\section{Notations and preliminary remarks}\label{sec:Preliminaries}
Take a -- possibly non-Gaussian -- one-dimensional \Levy\ process
$L(\cdot)$, and the Ornstein-Uhlenbeck () process $X(\cdot)$
solution of the stochastic differential equation (\sde)
            \begin{equation}\label{eq:genOU_sde}
              dX(t) =  -bX(t)dt + dL(t) \quad\qquad X(0)=X_0\quad \Pqo\qquad
              b>0
            \end{equation}
to wit
\begin{equation}
X(t) = X_0\,e^{-bt} + Z(t) \qquad\quad Z(t)=\int_0^t
e^{-b\,(t-s)}dL(s) \label{eq:sol:OU}
\end{equation}
Hereafter $L(\cdot)$ will be called \emph{background driving \Levy\
process} ({BDLP}) and $L(t)$ will also represent its stationary
increment of width $t$ that completely define the process, but for
an arbitrary initial condition. It is known that the \chf\
(\emph{characteristic function}) $\varphi_L(u,t)$ of these
increments, and their \lch\ (\emph{logarithmic characteristic})
$\psi_L(u,t)=\ln\varphi_L(u,t)$ are retrievable from a given \id\
(\emph{infinitely divisible}) law with \chf\
$\varphi_L(u)=e^{\psi_L(u)}$ according to
\begin{equation*}
   \varphi_L(u,t)=\varphi_L(u)^{t/T}\qquad\qquad\psi_L(u,t)=\frac{t}{T}\psi_L(u)
\end{equation*}
Here $T$ represents an arbitrary constant time scale introduced to
keep a fair balance among the physical dimensions: for practical
purposes however it is possible to take $T=1$, as we will do later
on in the present paper. Note that here the given \id\
$\varphi_L(u)$ and $\psi_L(u)$ can also be considered as shorthand
notations for $\varphi_L(u,T)$ and $\psi_L(u,T)$, namely the
characteristics of the \rv\ (\emph{random variable}) $L(T)$ famously
eponym of our {BDLP}.

Following a Barndorff-Nielsen and Shephard\mycite{BNSh01}
convention, if $\overline{\mathcal{D}}$ is the law of the stationary
process we will say that $X(\cdot)$ is a $\overline{\mathcal{D}}$-OU
process; when on the other hand the \rv\ $L(T)$ is distributed
according to the \id\ law ${\mathcal{D}}$ we will say that
$X(\cdot)$ is an OU-${\mathcal{D}}$ process. A well-known result
(see for instance Cont and Tankov~\cite{ContTankov2004} or
Sato~\cite{Sato}) states that a distribution
$\overline{\mathcal{D}}$ can be the stationary law of a given
OU-${\mathcal{D}}$ process if and only if $\overline{\mathcal{D}}$
is \emph{self-decomposable} (\sd, see more below). In addition, just
by taking an arbitrary degenerate initial condition $X_0=x_0,\;\Pqo$
-- and if we can manage to retrieve the distribution of its second,
integral term $Z(t)$ -- from the pathwise solution\refeqq{eq:sol:OU}
it is also apparently possible to deduce the transition \pdf\
(\emph{probability density function}) of the Markov process
$X(\cdot)$, and therefore all its distributional details in an
explicit form. It is appropriate to point out moreover that in the
equation\refeqq{eq:sol:OU} we provided the solution in terms of the
original {BDLP} $L(t)$ rather than of the dimensionless time {BDLP}
$L(b\,t)$ as done in Barndorff-Nielsen and Shephard\mycite{BNSh01}:
therefore a few results of ours will turn out to be explicitly
dependent on the parameter $b$. The differences between these two
equivalent representations are also discussed in
Barndorff-Nielsen~\cite{BN98}, Barndorff-Nielsen and
Shephard~\cite{BNSh03a} or Schoutens~\cite{Schoutens03} page 48.

Going back now to the \sde\refeqq{eq:genOU_sde} it is possible to
see (see also Barndorff-Nielsen et al.\mycite{BJS1998}) that the
solution process\refeqq{eq:sol:OU} is stationary if and only if its
\chf\ $\varphi_X(u,t)$ is constant in time and steadily coincides
with the \chf\ $\overline{\varphi}_X(u)$ of the (\sd) invariant
initial distribution that turns out to be decomposable according to
\begin{equation*}
\overline{\varphi}_X(u)=\overline{\varphi}_X(u\,e^{-b\,
t})\varphi_Z(u,t)
\end{equation*}
where now, at every given $t$, $\varphi_Z(u,t)=e^{\psi_Z(u,t)}$
denotes the \id\ \chf\ of the \rv\ $Z(t)$ in\refeqq{eq:sol:OU}. This
last statement apparently means that the law of $Z(t)$ in the
solution\refeqq{eq:sol:OU} coincides with that of the \arem\ of the
\sd, stationary law $\overline{\varphi}_X$ provided that $a=e^{-b\,
t}$, and that moreover we have
\begin{eqnarray}
    \varphi_Z(u,t)&=&\frac{\overline{\varphi}_X(u)}{\overline{\varphi}_X(u\,e^{-b\,
    t})}\label{arem}\\
    \psi_Z(u, t) &=& \overline{\psi}_X(u) - \overline{\psi}_X(u\,e^{-b\,
    t}). \label{eq:ou:cumulants}
\end{eqnarray}
It turns out therefore that studying the transition law of an OU
process essentially amounts to find first its stationary law, and
then the law of its \arem\refeqq{arem}: it is easy indeed to see
from\refeqq{eq:sol:OU} that the \chf\ of the time homogeneous
transition law with a degenerate initial condition $X(0)=x_0,\;\Pqo$
is
\begin{equation}\label{transchf}
    \varphi_X(u,t|x_0)=e^{\,ix_0ue^{-bt}}\varphi_Z(u,t)=\frac{\overline{\varphi}_X(u)\,e^{\,ix_0ue^{-bt}}}{\overline{\varphi}_X(u\,e^{-b\, t})}
\end{equation}
As a consequence one can focus on the properties of the \arem s of
these \sd\ distributions in order to deduce also the transition
\pdf\ of the associated OU processes. On the other hand, since the
law of an \arem\ also is \id, one could even construct its
associated \Levy\ process resulting in an even wider range of
possible models (see for instance Gardini et
al.\mycite{Gardini20b, Gardini20a}).

A number of additional relations between the distribution of the
stationary process, and that of the \rv\ $L(T)$ are known: it is
possible to show for instance that between the \lch's
$\overline{\psi}_X(u)=\ln\overline{\varphi}_X(u)$ of the stationary
distribution, and $\psi_L(u)=\ln\varphi_L(u)$ of $L(T)$ the
following relation holds (see Taufer and Leonenko\mycite{TAUFER2009}
and Schoutens\mycite{Schoutens03})
\begin{equation}\label{statlch}
\overline{\psi}_X(u) = \frac{1}{T}\int_0^{+\infty}
\psi_L(u\,e^{-bs}) ds
\end{equation}
On the other hand, assuming for simplicity that the \Levy\ measure
of $L(T)$ and that of the stationary process admit densities --
respectively denoted as $\nu_L(x)$ and $\overline{\nu}_X(x)$ -- and
supposing that $\overline{\nu}_X(x)$ is differentiable, it also
results (see Sato\mycite{Sato}, Cont and
Tankov\mycite{ContTankov2004})
\begin{align}
  &  \overline{\nu}_X(x) = \frac{U(x)}{Tb\,|x|}
  \qquad\qquad  U(x)=\left\{\begin{array}{ll}
                              \int_{-\infty}^x\nu_L(y)dy & x<0 \\
                              \int_x^{+\infty}\nu_L(y)dy & x>0
                              \end{array}
                     \right.\label{U}   \\
  & \overline{\nu}_X(x) +x
  \overline{\nu}_X '(x)=-\frac{\nu_L(x)}{Tb}\qquad\qquad x\neq0\label{eq:levy:measures:ou}
\end{align}
Taking advantage finally
of\refeqq{eq:ou:cumulants},\refeqq{transchf} and\refeqq{statlch} it
is easy to see that the transition \lch\ of the OU process can also
be written in terms of the corresponding $\psi_L(u)$ of $L(T)$ in
the form
\begin{equation}
    \psi_X(u,t|x_0)=iux_0e^{-bt} + \psi_Z(u,t) = iux_0e^{-b\, t} + \int_0^t\psi_L\left(ue^{-b\,s}\right)ds.
\label{eq:cumulant:function}
\end{equation}
As a consequence we can also calculate the cumulants
$c_{X,k}(x_0,t),\;k= 1,2,\ldots$ of $X(t)$ for $X_0=x_0$ from the
cumulants $c_{L,k}$ of $L(T)$ according to
\begin{eqnarray}
c_{X,1}(x_0,t) &=&\EXP{X(t)|X_0=x_0}\;=\; x_0e^{-b\, t} + \frac{c_{L,1}}{b\,T}\left(1-e^{-b\, t}\right),\qquad k=1\label{eq:cumulants:ou1}\\
c_{X,k}(x_0,t) &=& \frac{c_{L,k}}{k\,b\,T}\left(1-e^{-k\,b\,
t}\right), \qquad\qquad\qquad\qquad\qquad\qquad
k=2,3,\ldots\label{eq:cumulants:ou2}
\end{eqnarray}
On the other hand according to\refeqq{eq:ou:cumulants} the said
cumulants $c_{X,k}(x_0,t),\;k= 1,2,\ldots$ for $X_0=x_0$ can also be
derived from those of the stationary law here denoted
$c_{\overline{X},k}$
\begin{eqnarray}
c_{X,1}(x_0,t) &=&\EXP{X(t)|X_0=x_0}\;=\; x_0e^{-b\, t} + c_{\overline{X},1}\left(1-e^{-b\, t}\right),\qquad k=1\label{eq:cumulants:ou3}\\
c_{X,k}(x_0,t) &=& c_{\overline{X},k}\left(1-e^{-k\,b\,t}\right),
\qquad\qquad\qquad\qquad\qquad k=2,3,\ldots\label{eq:cumulants:ou4}
\end{eqnarray}
These quantities can be used both as benchmarks to test the
performance of the simulation algorithms, and to carry out an
estimation procedure based on the generalized method of moments.

By summarizing (see also Sabino\mycite{Sabino20b}), determining the
transition law of the OU processes $X(\cdot)$ consists of two steps,
that of course can also be used to produce a simulation algorithm:
\begin{itemize}
    \item given the {BDLP} $L(\cdot)$, find the the stationary distribution
    $\overline{\varphi}_X(u)$ of\refeqq{eq:genOU_sde};
    \item from\refeqq{arem} find the distribution of its \arem\ $Z(\cdot)$ with $a = e^{-b\,t}$.
\end{itemize}
In particular, the sequential generation of the process skeleton on
a time grid $t_0, t_1, \dots, t_M$ will consist now in finding a
simulation algorithm for the \arem\ $Z(\cdot)$ of the stationary
law: assuming indeed at each step $a_i=e^{-b(t_{i}-t_{i-1})},\; i=0,
\dots, M$, we just implement the following recursive procedure with
initial condition $X(t_0)=x_0$:
\begin{equation}
X(t_{i}) = a_i X(t_{i-1}) + Z_{a_i}, \quad\qquad i=1,\dots, M.
\label{eq:ou:recursion}
\end{equation}
where from\refeqq{arem} $Z_{a_i}$ are \rv's with \chf's
\begin{equation*}
    \chi_i(u,t)=\frac{\overline{\varphi}_X(u)}{\overline{\varphi}_X(u\,e^{-b(t_{i}-t_{i-1})})}
\end{equation*}
The previous equations suggest in some way two procedures to derive
the properties of the transition law of an OU process: the first
makes a start from its stationary law, the second from the law of
its {BDLP}. In the following sections we will explore both
directions focusing our attention on the \emph{tempered stable} laws
({TS}; see Rosi\'nski\mycite{ROSINSKI2007677}) with \emph{finite
variation} (for details see Cont and Tankov\mycite{ContTankov2004}),
and we will analyze both the {TS}-OU and the OU-{TS} processes. To
this end, we recall that the finite variation {TS} laws have \Levy\
densities of the form
\begin{equation}
\nu(x) = c\,\frac{q(x)}{|x|^{1+\alpha}}\qquad \quad c>0, \quad
0\le\alpha<1 \label{eq:ts:levy}
\end{equation}
where the tempering term $q(x)$ with $q(0)=1$ is monotonically
decreasing and $q(+\infty)=0$ for $x>0$, and monotonically
increasing and $q(-\infty)=0$ for $x<0$. We do not adopt however the
full characterization of Rosi\'nski\mycite{ROSINSKI2007677} because
we will focus on one-dimensional laws only, and mainly on the
exponentially modulated {TS}, also known as \emph{classical tempered
stable} laws ({CTS}), where $q(x)=e^{-\beta_1 x}$ for $x\ge 0$,
while $q(x)=e^{\beta_2 x}$ for $x< 0$ with $\beta_1, \beta_2>0$.


\section{The TS distributions and their \arem s}\label{sec:tsou}

In the forthcoming sections we will discuss both the TS-OU and the
OU-TS processes looking in particular to the properties of the
\arem\ of their stationary laws, and we will focus our attention
chiefly on the CTS subfamily. It is worthwhile noticing first that
the study of these processes has extensively been carried on in the
literature and that several types of TS laws have been investigated.
For instance, Kawai and Masuda\mycite{KawaiMasuda_1, KawaiMasuda_2}
and Zhang\mycite{Zhang11} have considered OU processes whose
stationary marginal law is a CTS distribution, whereas Bianchi et
al.\mycite{BRF17} assume a rapidly decreasing TS law (RDTS) and
Grabchak\mycite{Grabchak20} finally harmonizes all these types of
laws considering the larger class of general TS distributions.
Albeit many results can consequently be found in the literature
cited so far, we will nevertheless elaborate a little on this topic
also to show how the proofs of the propositions can be carried out
in a simple way by taking advantage of the properties of the \arem s
and of their \Levy\ triplets. For this purpose let us remember in
particular that, as it is well-known, the \sd\ laws constitute a
subclass of the class of the \id\ distributions having an
absolutely-continuous \Levy\ measure with density
\begin{equation*}
    {\nu}(x)=\frac{k(x)}{|x|}
\end{equation*}
where $k(x)$ is increasing in $(-\infty, 0)$ and decreasing in $(0,
+\infty)$ (see  Cont and Tan\-kov~\cite{ContTankov2004}, Proposition
15.3). Remark then that every TS distribution
satisfying\refeqq{eq:ts:levy} also is \sd. The law of the \arem\ of
a \sd\ law is \id\ too (see Sato \cite{Sato}) and the following
proposition characterizes it in terms of its \Levy\ triplet.
\begin{prop}\label{prop:sdpp}
Consider a \sd\ law  with  \Levy\ triplet $(\gamma, \sigma, \nu)$,
then for every $0<a<1$ the law of its \arem\ has \Levy\ triplet
$(\gamma_a, \sigma_a, \nu_a)$:
    \begin{eqnarray}
    \gamma_a &=& \gamma(1-a) - a\int_{\mathbb{R}}\sign(x) (\mathds{1}_{|x|\le \frac{1}{a}} - \mathds{1}_{|x|\le 1})k(x)\,dx \\
    \sigma_a &=& \sigma\sqrt{(1-a^2)} \\
     \nu_a(x) &=& \frac{k(x)-k(^x/_a)}{|x|}\;=\;\nu(x)-\frac{\nu\left(^x/_a\right)}{a}\label{eq:levy:law++}
    \end{eqnarray}
\end{prop}
\begin{proof}
The \LKh\ representation of the \lch\ of our \sd\ law can be given
in two equivalent ways just by redefining the drift term:
    \begin{eqnarray}
    \psi(u) &=& \left\{
    \begin{array}{l}
         iu\gamma - \frac{1}{2}\sigma^2u^2 + \int_{\mathbb{R}}\left(e^{iux}-1 - iux\mathds{1}_{|x|\le 1}\right)\frac{k(x)}{|x|}dx
         \\
         iu\gamma' - \frac{1}{2}\sigma^2u^2 + \int_{\mathbb{R}}\left(e^{iux'}-1 - iux'\mathds{1}_{|x'|\le\frac{1}{a}}\right)\frac{k(x')}{|x'|}dx'
    \end{array}
    \right.\label{eq:laplace:sub++}\\
    \gamma' &=& \gamma + \int_{\mathbb{R}}\sign (x)(\mathds{1}_{|x|\le \frac{1}{a}} - \mathds{1}_{|x|\le
    1})\,k(x)\,dx\nonumber
    \end{eqnarray}
Therefore, using both the representations\refeqq{eq:laplace:sub++}
with the change of variable $x = ax'$ in the second integral, the
\lch\ of the \arem\ $\psi_a(u)=\psi(u) - \psi(au)$ becomes
    \begin{eqnarray*}
        \psi_a(u) &=& iu\gamma(1-a)  - a\,\int_{\mathbb{R}}\sign (x)(\mathds{1}_{|x|\le \frac{1}{a}} - \mathds{1}_{|x|\le 1})\,k(x)\,dx\\
        &&\qquad\qquad-\frac{\sigma^2(1-a^2)u^2}{2}+\int_{\mathbb{R}}\left(e^{iux}-1 - iux\mathds{1}_{|x|\le 1}\right)\frac{k(x) - k(^x/_a)}{|x|}dx
    \end{eqnarray*}
Due to the properties of $k(x)$ it turns out that $k(x)-k(^x/_a)>0$
for every $x$ and for every $0<a<1$; and, as it also happens that
with  $\nu_a(x)$ defined in\refeqq{eq:levy:law++} we have
\begin{equation*}
    \int_{\mathbb{R}}(1\wedge x^2)\nu_a(x)dx<+\infty
\end{equation*}
it is easy to see that $\nu_a(x)$ qualify as a \Levy\ measure. Then
$(\gamma_a, \sigma_a, \nu_a)$ represents the legitimate \Levy\
triplet of the law of the \arem\ of a \sd\ law.
\end{proof}
In the context of the OU processes where the law of $Z(t)$
in\refeqq{eq:sol:OU} at the time $t$ is the \arem\ of the stationary
law for $a=e^{-bt}$, the Proposition~\ref{prop:sdpp} along with the
equation\refeqq{U} enables us to connect the \Levy\ density
$\nu_Z(x,t)$ of the \id\ \rv\ $Z(t)$ at the time $t$ to
$\overline{\nu}_X(x)$ and $\nu_L(x)$, the \Levy\ densities
respectively of the \sd\ stationary law and of the BDLP $L(T)$ at
time $T$:
\begin{eqnarray}
    \nu_Z(x,t) &=&
    \frac{\overline{k}_X(x) - \overline{k}_X(^x/_a)}{|x|}\;=\; \overline{\nu}_X(x) - \frac{\overline{\nu}_X(^x/_a)}{a}
               \qquad\qquad\qquad a = e^{-b\, t}
    \nonumber\\
     &=& \frac{U(x)-U\left(^x/_a\right)}{Tb\,|x|}\;=\;\frac{1}{Tb\,|x|}
              \left\{
                   \begin{array}{ll}
                   \int_{\,^x\!/_a}^x \nu_L(y)dy   & \;\hbox{$x<0$} \\
                                 \\
                   \int_x^{\,^x\!/_a} \nu_L(y)dy   & \;\hbox{$x>0$}
                   \end{array}
             \right.
    \label{eq:levy:measures:ou++_Z}
\label{eq:levy:measures:ou++}\label{eq:levy:measures:ou++_X}
\end{eqnarray}
According to the previous representations we can therefore adopt one
of two possible strategies to study the properties of the transition
law of a \Levy-driven OU process: the first based on the stationary
law and more suitable for a $\overline{\mathcal{D}}$-OU process; the
second using the distribution of the BDLP and more suitable for an
OU-$\mathcal{D}$ process.

A \Levy\ process is said to be of \emph{finite variation} when its
trajectories are of finite variation with probability $1$, and it is
possible to prove (Cont and Tankov~\cite{ContTankov2004}) that this
happens if and only if its characteristic triplet $(\gamma, \sigma,
\nu)$ satisfies the conditions
\begin{equation}
     \sigma=0, \qquad\qquad
     \int_{|x|\le1}x\,\nu(x)\,dx< +\infty\label{fvdef}
\end{equation}
An important subclass of such processes is that of
\emph{subordinators} that are \Levy\ processes with almost surely
non-decreasing sample paths: in this event their L\'evy triplet must
satisfy the conditions
\begin{equation}
    \sigma=0, \qquad\qquad\nu(x)=0\qquad x<0,\qquad\qquad
     \int_{0}^{+\infty}(x\wedge 1)\,\nu(x)\,dx< +\infty\label{subdef}
\end{equation}
As a matter of fact it would also be easy to see from the
L\'evy-Khintchin characterization theorem that any process of finite
variation can be written as the difference of two independent
subordinators -- for instance the Variance Gamma (VG) processes can
be represented as the difference of two Gamma processes -- and
therefore, without loss of generality, we can focus our attention on
subordinators only. Without presuming now that the actual processes
involved are of finite variation, subordinators or even L\'evy
processes, we will say hereafter for short that an \id\ law with
L\'evy triplet $(\gamma, \sigma, \nu)$ is of finite variation when
it satisfies the conditions\refeqq{fvdef}, and is a subordinator
when it satisfies the conditions\refeqq{subdef}. Finally we will
assume for simplicity $\gamma=0$ because a non zero value would
correspond just to a constant shift.


\begin{prop}\label{prop:sub_hom}
Consider a TS law with \Levy\ triplet $(0,0, \nu)$ such that
$\nu(x)=0$ for $x<0$ and
    \begin{equation}
    \nu(x)=c\,\frac{q(x)}{x^{1+\alpha}}\quad\qquad x>0,\qquad c\ge 0, \qquad 0\le\alpha< 1
    \label{eq:hom}
    \end{equation}
\noindent where $q(x)$ is a tempering function such that  $q(x) - q(\gamma\, x)=o(x^{\alpha})$, $x\rightarrow 0^+$ for every $\gamma>0$. Then the \Levy\ density of its \arem\ is
    \begin{equation}
    \nu_a(x) = \nu_1(x) + \nu_2(x),\quad\left\{
                                          \begin{array}{l}
                                            \nu_1(x)=c\,(1-a^\alpha)\frac{q(x)}{x^{1+\alpha}} \\
                                              \\
                                            \nu_2(x)=c\,a^\alpha\,\frac{q(x)\,-\,q(\,^x\!/_a)}{x^{1+\alpha}}
                                          \end{array}
                                        \right.
    \quad\lambda_a=\!\!\int_0^{+\infty}\!\!\!\!\nu_2(x)dx<+\infty
    \label{eq:levy:hom_arem}
    \end{equation}
\end{prop}
\begin{proof}
As already remarked every TS is \sd\ and, if confined to $x>0$, it
is easy to see that\refeqq{subdef} is satisfied so that the law is a
subordinator and hence also of finite variation. As a consequence
there is an \arem\ and from Proposition\myref{prop:sdpp} we have
\begin{eqnarray*}
\nu_a(x) &=& \nu(x)-\frac{\nu(\,^x/_a)}{a}\;=\;(1 - a^{\alpha}) \nu(x) + a^{\alpha}\nu(x) - \frac{\nu(\,^x/_a)}{a} \\
&=& c\,(1-a^\alpha)\frac{q(x)}{x^{1+\alpha}} + c a^{\alpha}
\frac{q(x) - q(\,^x/_a)}{x^{1+\alpha}}\;=\;\nu_1(x)+\nu_2(x)
\end{eqnarray*}
Now, while $\nu_1(x)$ is just a rescaled form of the original TS
L\'evy measure, we find that $\nu_2(x)>0$ is also normalizable. From
the properties of $q(x)$ we have indeed that
\begin{equation*}
\frac{q(x) - q(\,^x/_a)}{x^{1+\alpha}}=o\,(x^{-1-\alpha}),\qquad\quad
x\rightarrow+\infty,\qquad\quad0<\alpha<1
\end{equation*}
while from our hypothesis  we have
\begin{equation*}
\frac{q(x) - q(\,^x/_a)}{x^{1+\alpha}}=o(x^{-1}), \qquad\quad x\rightarrow 0^+.
\end{equation*}
and the integral in\refeqq{eq:levy:hom_arem} turns out to be finite
\end{proof}
Under the conditions of the Proposition~\ref{prop:sub_hom} we have
found that $\nu_2(x)$ is an integrable non-negative function with
$\Lambda_a$ in\refeqq{eq:levy:hom_arem} playing the role of a
normalization constant, and therefore that
$g_a(x)=\nu_2(x)/\Lambda_a$ can be interpreted a full fledged \pdf.
As a consequence $\nu_2(x)$ can be considered as the L\'evy measure
density of a compound Poisson law of parameter $\Lambda_a$ and jump
length distributed according to the \pdf\ $g_a(x)$. Since on the
other hand $\nu_1(x)$ is the L\'evy measure density of the original
TS law but with the rescaled parameter $c(1-a^\alpha)$, according to
the previous proposition we can claim that the \arem\ $Z_a$ of a TS
law with L\'evy density\refeqq{eq:hom} is -- in distribution -- the
sum $Y_a+C_a$ of two independent \rv's: a TS $Y_a$ of the
type\refeqq{eq:hom} with parameter $c(1-a^\alpha)$, and a compound
Poisson \rv\
\begin{equation*}
    C_a=\sum_{k=1}^{N_a}J_{a,k}
\end{equation*}
where $N_s$ is a Poisson \rv\ of parameter $\Lambda_a$ and
$J_{a,k},\;k=1,2,\ldots$ are a sequence of \iid\ \rv's with \pdf\
$g_a(x)$

As can be seen from\refeqq{transchf} in the
Section~\ref{sec:Preliminaries}, the transition law of a generalized
OU process\refeqq{eq:sol:OU} directly follows from the \arem\ of its
\sd\ stationary distribution when we take $a=e^{-b\, t}$. Adding in
the results of the Proposition~\ref{prop:sub_hom} we can therefore
fully display the L\'evy measure of the said transition laws when
the stationary distribution is a TS subordinator with \Levy\
density\refeqq{eq:hom}. This class of laws is not without merits in
itself and is a relevant one-dimensional subfamily of the general
tempered stable laws discussed in Grabchak \mycite{Grabchak20}. It
includes on the other hand the CTS subordinators with $0\le\alpha<
1$, while in fact the Theorem 1 in Zhang\mycite{ZZhang07} for
Inverse Gaussian-OU processes (IG-OU) and the Theorem 1 in
Zhang\mycite{Zhang11} for TS-OU are special cases of the
Proposition~\ref{prop:sub_hom}. The proposition also covers the
$p\,$TS-OU introduced in Grabchak\mycite{Grabchak16}, the special
case of Rapidly Decreasing TS-OU (RDTS-OU) discussed in Kim et
al\mycite{KRSBF10} and Bianchi et al.\mycite{BRF17}, the Modified
TS-OU (MTS-OU) studied in Kim et al.\mycite{KRSBF09}, and the Bessel
TS-OU (BTS-OU) discussed in Chung\mycite{Chung16}. It is expedient
to notice at this point that, although for a fixed time $t$ the law
of $Z(t)$ defined in\refeqq{eq:sol:OU} is \id\ and coincides with
that of the \arem\ of the stationary law taking $a=e^{-b\,t}$, this
process is not \Levy\ because $a$ changes in time.

The particular case of a TS-OU process with $\alpha=0$ (namely a
$\Gamma$-OU process if the stationary law is a CTS) is also
noteworthy: in this event indeed the BDLP turns out to be just a
compound Poisson process. In fact it is easy to see that, for $x>0$
and $\alpha=0$, from\refeqq{U},\refeqq{eq:levy:measures:ou++_Z}
and\refeqq{eq:hom} it results
\begin{equation*}
\int_x^{+\infty}\nu_L(y)\,dy=U(x) = Tc\,b\,\,q(x)
\end{equation*}
and hence (since $q(0)=1$)
\begin{equation*}
\int_0^{+\infty}\nu_L(x)dx=Tcb<+\infty\qquad\quad
\nu_L(x)=-Tcb\,q'(x)
\end{equation*}
Therefore the BDLP will be a compound Poisson process
\begin{equation}\label{eq:cp:bdlp}
    L(t)=\sum_{k=1}^{N(t)}J_k
\end{equation}
where now $N(t)$ is a Poisson process with intensity $\lambda=c\,b$,
and $J_k$ are \iid\ jumps with \pdf\
\begin{equation}
    f_J(x)=\frac{\nu_L(x)}{\lambda}=-q'(x)\label{fj}
\end{equation}
(remember that $q(x)$ is supposed to be non increasing). Then the
pathwise solution\refeqq{eq:sol:OU} of the OU
equation\refeqq{eq:genOU_sde} becomes now
\begin{equation}
    X(t) =X_0e^{-bt} + Z(t)\qquad\quad Z(t)=\sum_{k=1}^{N(t)}J_k
e^{-b(t-\tau_k)}\label{alpha0}
\end{equation}
where $\tau_k$ represent the jumping times of the Poisson process
$N(t)$. Of course this representation is valid for any BDLP compound Poisson and not only for that in\refeqq{eq:cp:bdlp}.

This type of $Z(t)$ has also interesting financial applications
beyond the context of OU processes: it can indeed describe random
cash-flows occurring at random time-to-maturities with a rate of
return equal to $b$. Remark moreover that -- as observed by
Lawrance\mycite{L80} in the context of Poisson point processes --
for every $t>0$ we have
\begin{equation}
\sum_{k=1}^{N(t)}J_k e^{-b(t-\tau_k)}\;\eqd\; \sum_{k=1}^{N(t)}J_k
e^{-b\,t\,U_k} \label{eq:law}
\end{equation}
irrespective of the law of $J_k$,
where $U_k\sim\unif([0,1])$ are a sequence of \iid\ \rv's uniformly
distributed in $[0, 1]$.


\section{Finite variation CTS-OU processes}
In this section we focus on CTS-OU processes with finite variation, namely the subclass of OU processes with stationary distribution $\cts\big(\alpha, \beta,
c\big)$ with $0\le\alpha < 1$. Such a subclass is especially manageable and 
in this particular case the Proposition\myref{prop:sub_hom}
entails indeed the following result.
\begin{prop}\label{prop:ctsou}
The CTS-OU process $X(t)$ with initial condition $X(0)=X_0,\;\Pqo$,
and with stationary distribution $\cts\big(\alpha, \beta, c\big)$
whose \Levy\ density\refeqq{eq:ts:levy} with $0\le\alpha<1,\;x>0$
has the tempering function $q(x)=e^{-\beta\,x},\;\beta>0$, can be
represented as
\begin{equation}
X(t) \eqd X_0e^{-b\, t} + X_1 + X_2
 \label{eq:prop:ctsou}
\end{equation}
where $X_1$ is again a $\cts\big(\alpha, \beta,
c(1-a^{\alpha})\big)$ with $a=e^{-b\,t}$, and
\begin{equation}
X_2 \eqd\sum_{i=1}^{N_a}\tilde{J}_i
\end{equation}
is a compound Poisson \rv\ with parameter
\begin{equation}\label{eq:lambda:ctsou}
\Lambda_a = c\Gamma\left(1-\alpha\right)\frac{\beta^{\alpha}}{\alpha}(1-a^{\alpha})
\end{equation}
and \iid\ jumps $\tilde{J}_k$, independent from $N_a$ and
distributed according to the \pdf
\begin{equation}\label{eq:fj:cts:ou}
f_J(x)=\frac{\alpha}{a^{-\alpha} -1}
\int_1^{\frac{1}{a}} \frac{(\beta\,v)^{1-\alpha} x^{-\alpha}v^{\alpha-1}}{ \Gamma(1-\alpha)}dv
\end{equation}
\end{prop}
\begin{remark}
The law with \pdf\ $f_J(x)$ coincides with the DTS distribution of
Zhang\mycite{Zhang11}: we show here that this is in fact a mixture
of a gamma laws with parameters $(1-\alpha, \beta V)$ and a random
$V$ distributed according to the \pdf\
\begin{equation}
f_V(v) = \frac{\alpha}{a^{-\alpha}-1}v^{\alpha-1}, \qquad 1\le
v\le\,^1/_a.
\end{equation}
It is easy to verify moreover that
\begin{equation}\label{eq:mix:law:ctsou}
V \eqd \left(1 + \frac{(a^{-\alpha} - 1)\,U}{\alpha}\right)^{\frac{1}{\alpha}},
\end{equation}
where $U\sim\unif(0,1)$ is a uniform \rv, and therefore its
simulation can be based on standard routines. In particular, when
$\alpha=1/2$, $X(t)$ turns out to be an IG-OU process and its
simulation no longer requires now the acceptance-rejection methods
adopted in Zhang\mycite{ZZhang07} and in Qu et al.\mycite{QDZ20},
but can be based on the method illustrated in Michael et
al.\mycite{MSH76} (see also Devroye\mycite{Dev86} page 148).
\end{remark}
\begin{proof}
The law of $Z(t)$ in the pathwise solution\refeqq{eq:sol:OU}
coincides with that of the \arem\ $Z_a$ of the stationary law
$\cts\big(\alpha, \beta, c\big)$ whose \Levy\ density is
\begin{equation*}
    \overline{\nu}_X(x)=c\,\frac{e^{-\beta\,x}}{x^{\alpha+1}}\qquad x>0,\quad
    a=e^{-b\,t}.
\end{equation*}
From the Proposition\myref{prop:sub_hom} we then have
\begin{equation*}
\nu_Z(x, t) =\nu_1(x,t) + \nu_2(x,t)\qquad\left\{
                                           \begin{array}{l}
       \nu_1(x)=\frac{c\,(1-a^{\alpha})}{x^{\alpha+1}}\,e^{-\beta\,x} \\
       \nu_2(x)=\frac{c\,a^{\alpha}}{x^{\alpha+1}}\big( e^{-\beta\,x}-e^{-\frac{\beta\,x}{a}}\big)
                                           \end{array}
                                         \right.
\end{equation*}
where $\nu_1(x)$ apparently corresponds to a $\cts\left(\alpha,
\beta, c(1-a^{\alpha})\right)$ law. As for the second term we have
(see Gradshteyn and Ryzhik\mycite{gradshteyn2007}, 3.434.1)
\begin{equation*}
\Lambda_a = c\,a^{\alpha}\int_0^{\infty}\frac{e^{-\beta\,x} -
e^{-\frac{\beta\,x}{a}}}{x^{\alpha+1}}dx=
\frac{c\,a^{\alpha}}{\alpha}\Gamma\left(1-\alpha\right)\beta^{\alpha}(a^{-\alpha}-1).
\end{equation*}
and therefore $\nu_2(x)$ is associated to the law of a compound
Poisson \rv\ with parameter $\Lambda_a$ and jumps distributed
according to $f_J(x)=\nu_2(x)/\Lambda_a$. On the other hand, since
\begin{equation*}
e^{-\beta\,x} - e^{-\frac{\beta\,x}{a}} = \,\int_0^{\frac{1}{a}} \beta\,x\,e^{-\beta\,v\, x}dv
\end{equation*}
we can also write
\begin{equation*}
f_J(x) = \frac{\alpha\, x^{-\alpha-1}}{\beta^{\alpha}(a^{-\alpha}
-1)\Gamma\left(1-\alpha\right)} \left( e^{-\beta\,x} -
e^{-\frac{\beta\,x}{a}}\right)= \frac{\alpha}{a^{-\alpha} -1}
\int_1^{a} \frac{(\beta\,v)^{1-\alpha} x^{1-\alpha}v^{\alpha-1}}{
\Gamma(1-\alpha)}dv
\end{equation*}
and this concludes the proof.
\end{proof}

\section{Finite Variation OU-CTS processes}\label{sec:outs}

In this section we will consider \Levy-driven OU processes whose
BDLP is a CTS process with $q(x)=e^{-\beta x}$. In this specific
case, the stationary law is not known in an explicit form (see for
instance Table 2 in Barndorff-Nielsen and Shephard\mycite{BNSh03a}),
but according to our discussion in the
Section\myref{sec:Preliminaries} the transition law of the
solution\refeqq{eq:sol:OU} of the equation\refeqq{eq:genOU_sde} can
nevertheless be retrieved through the formula\refeqq{transchf} if
the law of $Z(t)$ is known (remember that $Z(t)$ remains the same
for every initial condition). On the other hand we have shown that
the formula\refeqq{eq:levy:measures:ou++_Z} enables us to deduce the
\Levy\ measure density of $Z(t)$ at a given $t$ directly from the
\Levy\ measure density of the BDLP. We will show thus in the present
section that the transition law of our OU-CTS process is the
convolution of a CTS law (with parameters different from that of the
BDLP) and a compound Poisson law. In the following, we will denote
$\cts(\alpha, \beta, c)$ a CTS law with a \Levy\
density\refeqq{eq:hom} and $q(x)=e^{-\beta x}$.
\begin{prop}\label{prop:ou:ts}
For $0\le \alpha < 1 $, and at every $t>0$, the pathwise
solution\refeqq{eq:sol:OU} of an OU-CTS
equation\refeqq{eq:genOU_sde} with $X(0)=X_0,\;\Pqo$ is in
distribution the sum of three independent \rv's
\begin{equation}
X(t)\, \eqd\, aX_0+X_1 + X_2 \qquad\quad a=e^{-b\,
t}\label{eq:prop:outs_1}
\end{equation}
where $X_1$ is distributed according to the law $\cts\!\left(\alpha,
\frac{\beta}{a}, c\,\frac{1 - a^\alpha}{T\alpha\, b}\right)$, while
\begin{equation*}
X_2=\sum_{k=1}^{N_a} J_k
\end{equation*}
is a compound Poisson \rv\ where $N_a$ is an independent Poisson
\rv\ with parameter
\begin{equation}
 \Lambda_a =
 \frac{c\,\beta^\alpha\Gamma(1-\alpha)}{Tb\,\alpha^2a^\alpha}\,\left(1-a^\alpha+a^\alpha\log a^\alpha\right) \label{eq:outs:intensity}
\end{equation}
and $J_k$ are \iid\ \rv's with \pdf\
\begin{equation}
 f_J(x)=\frac{\alpha\, a^\alpha}{1-a^\alpha+a^\alpha\log a^\alpha} \int_1^{\frac{1}{a}}\frac{x^{-\alpha} \left(\beta\,
  v\right)^{1-\alpha}e^{-\beta v\, x}}{\Gamma(1-\alpha)}\, \frac{v^\alpha-1}{v}\, dv \label{eq:outs:jumps}
\end{equation}
\end{prop}
\begin{remark}\label{rem:ou:cts:mixture}
The \pdf\ $f_J(x)$\refeqq{eq:outs:jumps} can be seen as a mixture of
the gamma laws $\gam(1-\alpha,\beta V)$ with a random rate parameter
$V$ distributed according to the \pdf\
\begin{equation}
  f_V(v) = \frac{\alpha\, a^\alpha}{1-a^\alpha+a^\alpha\log
    a^\alpha}\, \frac{v^\alpha-1}{v}\qquad\quad 1\le v\le\,^1/_a\label{eq:ou:cts:mixture:rate}
\end{equation}
which is correctly normalized. Proposition\myref{prop:ou:ts} covers
the case of a OU-gamma process illustrated in Qu et
al.\mycite{QDZ19} when $\alpha$ tends to zero: we have indeed
\begin{equation*}
    \lim_{\alpha\rightarrow 0^+} \Lambda = \lim_{\alpha\rightarrow 0^+}\frac{c\,\beta^\alpha\Gamma(1-\alpha)}{Tb\,\alpha^2a^\alpha}\,\left(1-a^\alpha+a^\alpha\log a^\alpha\right)
    =  \frac{c\,\log^2 a}{2\,Tb}
\end{equation*}
and therefore, replacing $a=e^{-b\, t}$, we retrieve the equation
4.11 in Qu et al.\mycite{QDZ19}. Similarly, for $\alpha\rightarrow
0^+$ $f_J(x)$ coincides with the equation 4.9 in Qu et
al\mycite{QDZ19}, and can be seen as the \pdf\ of an exponentially
distributed \rv\ $\gam(1,\beta V)$ with a random rate parameter with
the \pdf\ (see the proof of the Theorem 4.1 in Qu et
al.\mycite{QDZ19})
\begin{equation*}
    \lim_{\alpha\rightarrow 0^+} f_V(v) = \frac{2\log v}{v\log^2 a}\qquad\quad 1\le
    v\le\,^1/_a
\end{equation*}
\end{remark}
\begin{proof}
Based on equation\refeqq{eq:levy:measures:ou++_Z} and with the
change of variable $y=wx$, the \Levy\ density of the term $Z(t)$ in
the pathwise solution\refeqq{eq:sol:OU} of an OU-CTS process is
(remember that the coefficient $a=e^{-bt}$ is time dependent)
\begin{eqnarray*}
 \nu_Z(x,t) &=& \frac{c}{Tb\,x}\int_{x}^{\frac{x}{a}}\frac{e^{-\beta y}}{y^{\alpha + 1}}\,dy
             \;=\; \frac{c}{Tb}\int_{1}^{\frac{1}{a}}\frac{e^{-\beta wx}}{x^{\alpha + 1}\, w^{\alpha + 1}}\,dw \\
  &=& \frac{c\,e^{-\frac{\beta}{a}x}}{Tb\,x^{\alpha+1}}\int_{1}^{\frac{1}{a}}\frac{dw}{w^{\alpha + 1}} +
\frac{c}{Tb}\int_{1}^{\frac{1}{a}}\frac{e^{-\beta wx} -
e^{-\frac{\beta}{a}x}}{x^{\alpha + 1}\, w^{\alpha +
1}}\,dw\;=\;\nu_1(x) + \nu_2(x).
\end{eqnarray*}
The first term apparently is the \Levy\ density of a CTS law
$\cts\!\left(\alpha, \frac{\beta}{a}, c\,\frac{1 -
a^\alpha}{T\alpha\, b}\right)$ because it is easy to see that
\begin{equation*}
\nu_1(x) = \frac{c\left(1 - a^\alpha\right) e^{-\frac{\beta
x}{a}}}{Tb\,\alpha\, x^{1+\alpha}}
\end{equation*}
On the other hand $\nu_2(x)>0$ for every $x>0$ because $e^{-\beta
wx}- e^{-\frac{\beta}{a}x}>0$ when $0<w<\,\!^1/_a$, and moreover
with $v=aw$ we find (see 3.434.1 in Gradshteyn and
Ryzhik\mycite{gradshteyn2007})
\begin{eqnarray*}
\Lambda_a &=& \int_{0}^{\infty}\nu_2(x)\,dx =
\int_{0}^{\infty}\frac{c}{Tb}\,dx\int_{1}^{\frac{1}{a}}
      \frac{e^{-\beta wx}- e^{-\frac{\beta}{a}x}}{x^{\alpha + 1}\, w^{\alpha + 1}}\,dw\\
  &=& \frac{c}{Tb} \int_{1}^{\frac{1}{a}}\frac{dw}{w^{\alpha + 1}}
             \int_{0}^{\infty}\frac{e^{-\beta wx} -e^{-\frac{\beta}{a}x}}{x^{\alpha + 1}}dx
       =\frac{c}{Tb} \int_{1}^{\frac{1}{a}}\frac{dw}{w^{\alpha +
   1}}\frac{\beta^\alpha\Gamma(1-\alpha)}{\alpha}(a^{-\alpha}-w^\alpha)\\
  &=&\frac{c\,\beta^\alpha\Gamma(1-\alpha)}{Tb\,\alpha}\int_a^1\frac{1-v^\alpha}{v^{1+\alpha}}\,dv
        =\frac{c\,\beta^\alpha\Gamma(1-\alpha)}{Tb\,\alpha^2a^\alpha}\,\left(1-a^\alpha+a^\alpha\log a^\alpha\right)
\end{eqnarray*}
where apparently $0<\Lambda_a<+\infty$. As a consequence
\begin{equation}\label{fJint}
 f_J(x)=\frac{\nu_2(x)}{\Lambda_a}=\frac{\alpha^2a^\alpha}{\left(1-a^\alpha+a^\alpha\log a^\alpha\right)\beta^{\alpha}\Gamma(1-\alpha)}
          \int_1^{\frac{1}{a}}\frac{e^{-\beta w\,x} - e^{-\frac{\beta x}{a}}}{x^{1+\alpha}w^{1+\alpha}}dw
\end{equation}
is a good \pdf\ and then, $\nu_2(x)$ represents the \Levy\ density
of a compound Poisson law with parameter $\Lambda_a$ and jumps
distributed according to the \pdf\ $f_J(x)$. It would be possible to
show now (see 3.381.3 in Gradshteyn and
Ryzhik\mycite{gradshteyn2007}) that
\begin{equation*}
 f_J(x)=\frac{\alpha^2a^\alpha}{\left(1-a^\alpha+a^\alpha\log
  a^\alpha\right)\Gamma(1-\alpha)}\left[\frac{\Gamma(-\alpha,\beta x)-\Gamma\left(-\alpha,\,^{\beta x}/_a\right)}{x}
   -\frac{1-a^\alpha}{\alpha\beta^\alpha}\,\frac{e^{-\frac{\beta x}{a}}}{x^{1+\alpha}}\right]
\end{equation*}
where $\Gamma(\gamma,z)$ is the incomplete gamma function. This
\pdf\ has a typical gamma-like behavior with
$f_J(x)=O(x^{-\alpha}),\;x\to0^+$ as can be seen from the
Figure\myref{fJ}.
\begin{figure}
\caption{Gamma-like behavior of the jumps \pdf\ $f_J(x)$ with
$\alpha=\,^1/_2,\;\beta=1,\;a=\,^1/_2$. The dashed curve shows that
$f_J(x)=O(x^{-\alpha}),\;x\to0^+$.}\label{fJ}
\includegraphics[width=9cm]{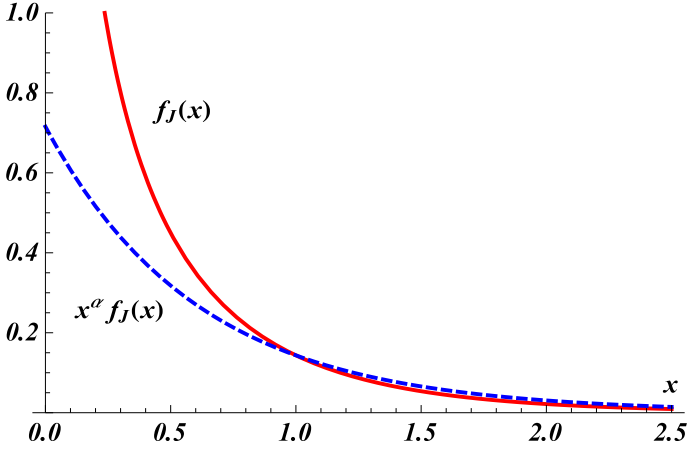}
\end{figure}
For later computational convenience however we prefer to give an
alternative representation of this jumps distribution. Since
\begin{equation*}
    e^{-\beta w\,x} - e^{-\frac{\beta x}{a}} =
      \beta\, x \int_w^{\frac{1}{a}}e^{-\beta vx }\,dv
\end{equation*}
and with an exchange in the order of the integrations, the
\pdf\refeqq{fJint} becomes
\begin{eqnarray*}
  f_J(x)&=&\frac{\alpha^2a^\alpha\beta^{1-\alpha}x^{-\alpha}}{\left(1-a^\alpha+a^\alpha\log a^\alpha\right)\Gamma(1-\alpha)}
          \int_1^{\frac{1}{a}}\frac{dw}{w^{1+\alpha}}\int_w^{\frac{1}{a}}e^{-\beta vx }\,dv
 \\
   &=&\frac{\alpha^2a^\alpha\beta^{1-\alpha}x^{-\alpha}}{\left(1-a^\alpha+a^\alpha\log a^\alpha\right)\Gamma(1-\alpha)}
          \int_1^{\frac{1}{a}}dv\,e^{-\beta vx}\int_1^w\frac{dw}{w^{1+\alpha}} \\
      &=&\frac{\alpha^2a^\alpha\beta^{1-\alpha}x^{-\alpha}}{\left(1-a^\alpha+a^\alpha\log
             a^\alpha\right)\Gamma(1-\alpha)}\int_1^{\frac{1}{a}}e^{-\beta vx}\,\frac{v^\alpha-1}{\alpha
v^\alpha}\,dv \\
     &=&\frac{\alpha\, a^\alpha}{1-a^\alpha+a^\alpha\log a^\alpha} \int_1^{\frac{1}{a}}\frac{x^{-\alpha} \left(\beta\,
  v\right)^{1-\alpha}e^{-\beta v\, x}}{\Gamma(1-\alpha)}\, \frac{v^\alpha-1}{v}\, dv
\end{eqnarray*}
that coincides with\refeqq{eq:outs:jumps} and is a mixture of gamma
laws $\gam(1-\alpha,\beta V)$ with a random rate parameter
distributed according to the \pdf\refeqq{eq:ou:cts:mixture:rate}, as
stated in the Remark\myref{rem:ou:cts:mixture}
\end{proof}

\noindent As discussed in the Section\myref{sec:Preliminaries}, the
knowledge of the law of $Z(t)=X_1+X_2$ in the
Proposition\myref{prop:ou:ts} also enables us to calculate the
distribution of the solution\refeqq{eq:sol:OU} of the given OU
equation\refeqq{eq:genOU_sde} with an arbitrary initial condition
$X_0$. In particular for a degenerate initial condition
$X_0=x_0,\;\Pqo$ we will have for every $t>0$
\begin{equation}
 X(t)\big|_{X_0=x_0}\, \eqd\,ax_0+ X_1 + X_2 \qquad\quad a=e^{-bt}\label{eq:prop:outs}
\end{equation}
where the distributions of $X_1,X_2$ are described in detail in the
Proposition\myref{prop:ou:ts}. Of course the
formula\refeqq{transchf} would give access then to the transition
distribution and therefore to all the details of the process.
\begin{remark}
The direct extension to the bilateral finite variation setting is
straightforward because every finite variation bilateral TS process
can be seen as the difference of two TS processes with different
parameters: as a consequence the Propositions\myref{prop:sdpp},
\myref{prop:ctsou} and \myref{prop:ou:ts}, can be easily extended to
the bilateral case. As far as the simulation of such processes is
concerned, this extension essentially boils down to running twice
the algorithms for TS subordinators. We omit an explicit proof to
avoid overloading the paper with lengthy details of routinary
nature.
\end{remark}
\section{Simulation Algorithms}\label{sec:sim}
The simulation of TS-OU processes and CTS laws have been widely
discussed in several studies (see for instance Kawai and
Masuda\mycite{KawaiMasuda_1, KawaiMasuda_2}, Zhang\mycite{Zhang11}
and Grabchack\mycite{Grabchak19, Grabchak20} and the references
therein) and several software packages are available for such
purpose. Therefore, in this section we will only illustrate how to
exactly simulate the CTS-OU and the OU-CTS processes that, although
similar in names, are two rather different objects as explained in
the previous sections.  As far as the CTS-OU processes are
concerned, our contribution is the enhancement of the simulation
performance by taking advantage of the $f_J(x)$
reprentation\refeqq{eq:fj:cts:ou} that, in contrast to
Zhang\mycite{Zhang11}, does not require any acceptance-rejection
procedure. On the other hand, with regard to OU-CTS processes, we
propose a new simulation procedure for the drawings from the mixture
with \pdf\refeqq{eq:ou:cts:mixture:rate}. At variance with the
approach of Qu et al.\mycite{QDZ20}, this algorithm is based on an
acceptance-rejection method whose expected number of iterations
before acceptance however can be made arbitrarily close to one and
is therefore more efficient. In our numerical experiments we
consider a time grid $t_0, t_1,\dots, t_M$, $\Delta t_m = t_m -
t_{m-1}\,,\; m=1,\dots,M$ with $M$ steps.

\subsection{CTS-OU processes}\label{sub:sim:oucts}
The simulation procedure for the generation of the skeleton of CTS
process is based on the Proposition\myref{prop:ctsou} and is
summarized in the Algorithm\myref{alg:ctsou}.
\begin{algorithm}
\caption{ }\label{alg:ctsou}
\begin{algorithmic}[1]
        \State $X_0\gets x$
        \For{ $m=1, \dots, M$}
        \State $a \gets e^{-b\Delta t_m}$
        \State $x_1\gets X_1\sim \cts\!\left(\alpha,\beta, c(1 - a^\alpha)\right)$
        \State $n\gets N_a\sim\poiss(\Lambda_a)$,
        \Comment{Generate an independent Poisson \rv\ with $\Lambda_a$ in\refeqq{eq:lambda:ctsou} }
        \State $u_i\gets U_i\sim\unif(0,1), i=1, \dots, n$
        \Comment{Generate $n$ \iid\ uniform \rv's}
        \State $v_i\gets\left(1 + \frac{a^{-\alpha}-1}{\alpha}u_i\right)^{\frac{1}{\alpha}}$
        \Comment{Generate according to\refeqq{eq:mix:law:ctsou}}
        \State $\tilde{\beta}_i \gets \beta\, v_i, i=1, \dots, n$
        \State $j_i\gets J_i\sim\gam(1-\alpha,\, \tilde{\beta}_i), i=1, \dots, n$
        \Comment{Generate $n$ independent gamma \rv's all with the same scale $1-\alpha$ and random rates}
        \State $x_2\gets \sum_{i=1}^nj_i$
        \State $X(t_m)\gets a\,X(t_{m-1}) + x_1 + x_2$.
        \EndFor
        \end{algorithmic}
\end{algorithm}
We remark that when $\alpha=0$ a CTS-OU process is a compound
Poisson process with a gamma stationary law whose efficient exact
simulation can be found in Sabino and Cufaro Petroni\mycite{cs20_1}.
\subsection{OU-CTS processes}\label{sub:sim:ctsou}
The simulation steps for the skeleton of a OU-CTS process are then
summarized in the Algorithm\myref{alg:ou}.
\begin{algorithm}
\caption{}\label{alg:ou}
\begin{algorithmic}[1]
        \State $X_0\gets x$
        \For{ $m=1, \dots, M$}
        \State $a \gets e^{-b\Delta t_m}$
        \State $x_1\gets X_1\sim \cts\!\left(\alpha,\frac{\beta}{a}, \frac{c(1 - a^\alpha)}{T\alpha\, b}\right)$
        \State $n\gets N_a\sim\poiss(\Lambda_a)$, \Comment{Generate an independent Poisson \rv\ with $\Lambda_a$ in\refeqq{eq:lambda:ctsou} }
        \State $v_i\gets V_i, i=1, \dots, n$ \Comment{Generate $n$ \iid\ \rv's with \pdf\ given by Equation\refeqq{eq:ou:cts:mixture:rate}}
        \State $\tilde{\beta}_i \gets \beta\, v_i, i=1, \dots, n$
        \State $j_i\gets J_i\sim\gam(1-\alpha,\, \tilde{\beta}_i), i=1, \dots, n$
        \Comment{Generate $n$ independent generalized gamma \rv's all with the same $p$, scale $p-\alpha$ and random rates}
        \State $x_2\gets \sum_{i=1}^nj_i$
        \State $X(t_m)\gets a\,X(t_{m-1}) + x_1 + x_2$.
        \EndFor
        \end{algorithmic}
\end{algorithm}
The sampling from a CTS law has been widely studied by several
authors (see for instance Devroye\mycite{Dev2009} and
Hofert\mycite{Hofert2012}), and here the only non-standard step is
the fifth one in the Algorithm\myref{alg:ou}, namely that allowing
the generation of the jumps of the compound Poisson process of
Proposition\myref{prop:ou:ts}. On the other hand, as mentioned in
Remark\myref{rem:ou:cts:mixture}, these jump sizes are \iid\
distributed \rv's following a gamma law with shape $1-\alpha$ and a
random rate. Therefore the unique remaining task is to sample from a
law with thr \pdf\ $f_V(v)$\refeqq{eq:ou:cts:mixture:rate}. Since
however this \pdf\ is not monotonic in $[1,\,^1/_a]$ for every value
of its parameters, we first define the new \rv
\begin{equation*}
    W=-\frac{\log V}{\log a}\qquad\qquad V=a^{-W}=e^{-W\log a}
\end{equation*}
that has now the \pdf\
\begin{eqnarray}
  f_W(w)&=&\frac{-a^\alpha\log a^\alpha}{1-a^\alpha+a^\alpha\log
    a^\alpha}\,\left(a^{-\alpha w}-1\right)\nonumber\\
 &=& \frac{\log a^{-\alpha}}{a^{-\alpha}-1-\log a^{-\alpha}}\,\left(e^{w\log a^{-\alpha}}-1 \right) \qquad\quad
           0\le w \le 1. \label{eq:jump:exp:mixture}
\end{eqnarray}
It is straightforward to check then that $f_W(w)$ is monotonic and
convex in $[0, 1]$ and hence one can rely on the inversion-rejection
algorithm illustrated in Devroye\mycite{Dev86} page 355. The
solution that we propose here is very similar to that, and in effect
consists in replacing the steps required for the sequential search
of the inversion part with the method of partitioning the densities
into intervals (see once again Devroye\mycite{Dev86} page 67). We
remark indeed that $f_W(w)$, besides being monotonic and convex,
also has the following upper bound
\begin{align*}
    & \qquad\qquad f_W(w)\le g(w) = G(a, \alpha)\, \bar{g}(w),\qquad\qquad\qquad 0\le w\le 1 \\
    & G(a, \alpha)=\frac{\log a^{-\alpha}\left(a^{-\alpha}-1\right)}{2\left(a^{-\alpha}-1-\log a^{-\alpha}\right)}\,,
            \qquad\quad \bar{g}(w) = 2\,w
\end{align*}
namely it is dominated by a linear function where $G(a, \alpha)$ is
the area under $g(x)$. We could therefore devise a simple
acceptance-rejection procedure where $G$ should be as close to $1$
as possible because it roughly represents the number of iterations
needed in the rejection algorithm. While however $G(a,
\alpha)\to1^+$ when $a\rightarrow 1^-$, unfortunately it is $G(a,
\alpha)\to+\infty$ for $a\rightarrow 0^+$. Taking therefore
$a=e^{-b\, t}$, this latter limit means that the generation of a
OU-CTS process with a large $b\, t$ might either have a heavy
computational cost, or potentially require a large number of
simulations.

In principle we could consider only small time steps, but on the
other hand the acceptance-rejection sampling can be easily improved,
via the modified decomposition method elucidated in
Devroye\mycite{Dev86} page $69$, just by taking a piecewise linear
dominating function $g(w)$. More precisely we partition $[0, 1]$
into $L$ disjoint intervals $\mathcal{I}_{\ell}=[w_{\ell-1},
w_{\ell}],\;\; \ell=1, \dots, L$,\;\;
$\bigcup_{\ell}\mathcal{I_{\ell}}=[0, 1]$ with $w_0=0$, and then we
have
\begin{align*}
    &\qquad\qquad\qquad\qquad\qquad f_W(w) \le g_L(w) = \sum_{\ell=1}^L \,g_{\ell}(w)\mathds{1}_\ell(w), \\
    & g_{\ell}(w) = \frac{f_W(w_{\ell}) - f_W(w_{\ell-1})}{w_{\ell} -
w_{\ell-1}}\,w + f_W(w_{\ell-1}),
                    \quad\mathds{1}_\ell(w)\!=\!\left\{
                                               \begin{array}{ll}
                                                 \!\!1, & \hbox{if $w\in\mathcal{I}_{\ell}$} \\
                                                 \!\!0, & \hbox{else}
                                               \end{array}
                                             \right.
    \quad \ell = 1, \dots, L
\end{align*}
where we can also write
\begin{align*}
    &\qquad\quad g_L(w) = G_L(a, \alpha)\, \sum_{\ell=1}^L p_{\ell}\,\bar{g}_{\ell}(w), \qquad\quad \bar{g}_{\ell} = \frac{g_{\ell}(w)}{q_{\ell}},\\
    & q_{\ell} =  \int_{\mathcal{I}_{\ell}}g_L(w)dw, \qquad\quad
p_{\ell} = \frac{q_{\ell}}{G_L(a, \alpha)}, \qquad\quad G_L(a,
\alpha) = \sum_{\ell=1}^L q_{\ell}.
\end{align*}
Apparently the $\bar{g}_{\ell}(w),\; \ell=1, \dots, L$ turn out to
be piecewise linear \pdf's, while the $p_\ell$ constitute a
discrete, normalized distribution. Increasing the number $L$ of the
intervals, with given $0 <a <1$ and $0<\alpha<1$, $G_L(a, \alpha)$
can be made arbitrary close to $1$ because it measures the
trapezoidal approximation of $\int_0^1f_W(w)dw=1$. On the other hand
the random drawing from the laws with \pdf's $\bar{g}_{\ell}(w),\;
\ell=1, \dots, L$ is very simple and can be implemented via the
standard routines. Denoting now with $S$ a \rv\ with distribution
$\PR{S=\ell}=p_{\ell},\; \ell=1, \dots, L$, and with $Y_{\ell}$ a
\rv\ with \pdf\ $\bar{g}_{\ell}(w)$, the
Algorithm\myref{alg:mixture:rate} summarizes the instructions needed
to implement the fifth step in the Algorithm\myref{alg:ou}. We
remark finally that an alternative procedure, leading to similar
results, might have been some shrewd decomposition of $f_W(w)$
rather than of its dominating curve. However Devroye\mycite{Dev86}
at the page 70 nicely spell out the reasons why the procedure here
adopted is in principle preferable.
\begin{algorithm}
\caption{}\label{alg:mixture:rate}
\begin{algorithmic}[1]
        \State $s \gets S$ \Comment{Generate a discrete \rv\ with $\PR{S=\ell}=p_{\ell},\; \ell=1, \dots, L$}
        \While{$u\le \frac{f_W(y)}{g_L(y)}$}
        \State $u\gets U\sim\unif{[0, 1]}$ \Comment{Generate a uniform \rv.}
        \State $y\gets Y_\ell$ \Comment{Generate a \rv\ with \pdf\ $\bar{g}_\ell(w)$}
        \EndWhile
        \State $v\gets a^{-\alpha\,y}$
        \State \Return v
        \end{algorithmic}
\end{algorithm}

\begin{remark}
It is worthwhile mentioning that an alternative procedure relying on
a different acceptance-rejection strategy has been proposed in Qu et
al.\mycite{QDZ20}. In contrast to this last approach, however, in
our algorithm $G_L(a, \alpha)$ can be made arbitrary close to $1$
irrespective of the value of $a$ (and of the size of the time-step),
and therefore our approach turns out to be computationally more
efficient.

On the other hand we can also take advantage of the interplay
between the OU-CTS and the CTS-OU processes to gain an insight into
the possible benefits of the different simulation strategies:
from\refeqq{eq:levy:measures:ou} we find indeed that the \Levy\
density of the BDLP $L(t)$ for a CTS-OU process is
\begin{equation}
\nu_L(x) = c\,T\,b\,\alpha\frac{e^{-\beta \, x}}{x^{\alpha + 1}} + c\,T\,b\,\beta\frac{e^{-\beta\,x}}{x^{\alpha}}.
\label{eq:levy:bdlp:ctsou}
\end{equation}
\noindent where the first term apparently provides the BDLP $L_1(t)$
of an OU-CTS, whereas the second one corresponds to a compound Poisson
$L_2(t)$ (see Cont and Tankov\mycite{ContTankov2004} page 132).
Therefore the path-wise solution\refeqq{eq:sol:OU} of our CTS-OU
process is now
\begin{equation*}
    X(t)=x_0e^{-bt}+Z_1(t)+Z_2(t)\qquad\left\{
                                         \begin{array}{l}
                                           Z_1(t)=\int_0^te^{-b(t-s)}dL_1(s) \\
                                           Z_2(t)=\int_0^te^{-b(t-s)}dL_2(s)
                                         \end{array}
                                       \right.
\end{equation*}
where $Z_1$ is a OU-CTS (with $Z_1(0)=0$) and, as shown
in\refeqq{alpha0}, $Z_2$ is a compound Poisson that can easily be
simulated based on\refeqq{eq:law}. On the other hand, according to
the Proposition\myref{prop:ou:ts}, the OU-CTS process $Z_1(t)$ is in
its turn the sum of a (time-dependent) CTS \rv\ $X_1$ and of a
(time-dependent) compound Poisson \rv\ $X_2$; so that ultimately a
CTS-OU process with a degenerate initial condition $X_0=x_0$ --
beyond being of the form\refeqq{eq:prop:ctsou} presented in the
Proposition\myref{prop:ctsou} -- can now be seen also as the sum of
four random terms: one distributed according to a CTS law, two
compound Poisson \rv's and a degenerate summand. Of course the two
representations coincide in distribution and, as a matter of fact,
the four-terms representation reproduces again that of Qu et
al.\mycite{QDZ20}; but the possible alternative simulation
algorithms stemming from the four term representation, although
perfectly correct, would require now the generation of three \rv's
and the use of acceptance-rejection methods in addition to that
needed for the sampling of a CTS distributed \rv, and therefore they
would turn out to be rather less efficient than the
Algorithm\myref{alg:ctsou}.

Finally, based on Remark\myref{rem:ou:cts:mixture} and on the
results of Qu et al.\mycite{QDZ19}, we notice that for $\alpha=0$
the simulation of $V$ is here much easier because no
acceptance-rejection method is required.
\end{remark}

\section{Numerical Experiments}\label{sec:num:exp}
In this section, we will assess the performance and the
effectiveness of our algorithms through extensive numerical
experiments. All the simulation experiments in the present paper
have been conducted using \emph{Python} with a $64$-bit Intel Core
i5-6300U CPU, 8GB. The performance of the algorithms is ranked in
terms of the percentage error relative to the first four cumulants
denoted \emph{err} \% and defined as
\begin{equation*}
   \text{err \%} = \frac{\text{true value} - \text{estimated
value}}{\text{true value}}
\end{equation*}
Finally, for simplicity we assume that the value of the constant
time scale is $T=1$.

\subsection{CTS-OU processes}\label{sub:num:ctsou}

Since the \Levy\ density  on $[0,+\infty)$ of the stationary law
$\cts\left(\alpha, \beta, c\right)$ of a CTS-OU process with finite
variation and parameters $\alpha, \beta, c$ is
\begin{equation*}
    \overline{\nu}_X(x)=\frac{c\,e^{-\beta x}}{x^{1+\alpha}}\qquad\quad
x>0,\qquad\beta>0,\qquad0\le\alpha<1
\end{equation*}
its cumulants are (see for instance Cont and
Tankov\mycite{ContTankov2004}, Proposition 3.13)
\begin{equation}\label{eq:cts:cumulants}
c_{\overline{X}, k} =\int_0^{+\infty}x^k\overline{\nu}_X(x)\,dx=
c\,\beta^{\alpha-k}\Gamma(k-\alpha)
\end{equation}
and therefore from\refeqq{eq:cumulants:ou3}
and\refeqq{eq:cumulants:ou4} we obtain the cumulants of
$X(\Delta t)$ with the degenerate initial condition $X_0=x_0$
\begin{equation}\label{eq:cumulants:ctsou}
    c_{X,k}(x_0,\Delta t)=x_0e^{-b\Delta t}\delta_{k,1} + c\,\beta^{\alpha-k}\Gamma(k-\alpha)(1-e^{-kb\Delta t})
    \qquad\quad k=1,2,\ldots
\end{equation}

In our numerical experiments we consider a CTS-OU process with
parameters $\left(b, c, \beta\right) = \left(10, 0.8, 1.4\right)$
whose trajectories with $\alpha\in\{0.3, 0.5, 0.7, 0.9\}$ are
displayed in Figure\myref{fig:ctsou:trajectories} where of course
the case $\alpha=0.5$ is that of an IG-OU (inverse Gaussian)
process. We remark that the sampling from an IG law can be performed
via the many-two-one transformation method of Michael et
al.\mycite{MSH76}, and therefore no acceptance-rejection procedure
is required in Algorithm\myref{alg:ctsou} to generate the skeleton
of an IG-OU process.
\begin{figure}
\caption{Sample trajectories of CTS-OU processes with $\left(b, c,
\beta\right) = \left(10, 0.8, 1.4\right)$ and $\alpha\in\{0.3, 0.5,
0.7, 0.9\}$}\label{fig:ctsou:trajectories}
        \begin{subfigure}[c]{.5\textwidth}{
                \includegraphics[width=70mm]{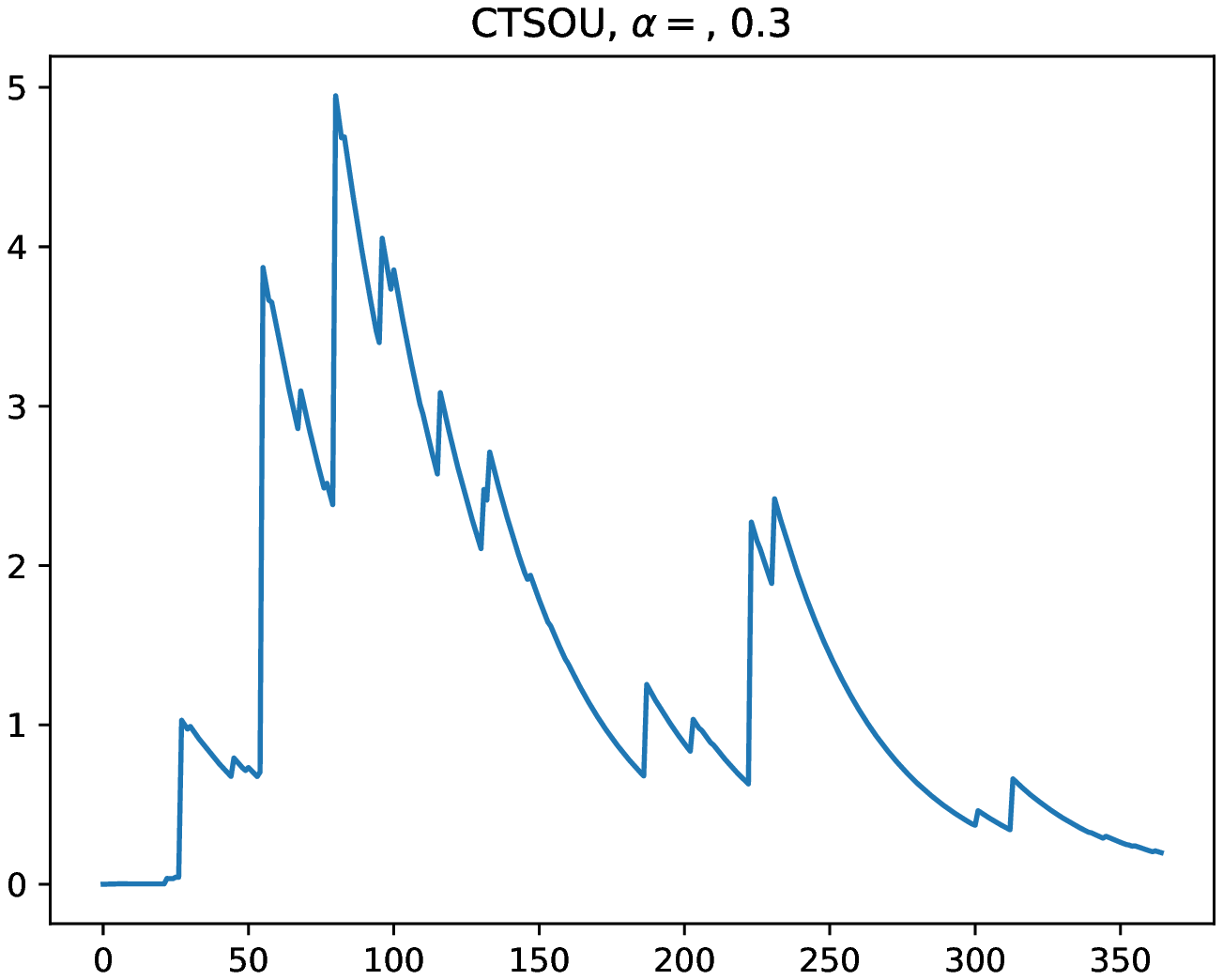}
                }
        \end{subfigure}
        \begin{subfigure}[c]{.5\textwidth}{
                \includegraphics[width=70mm]{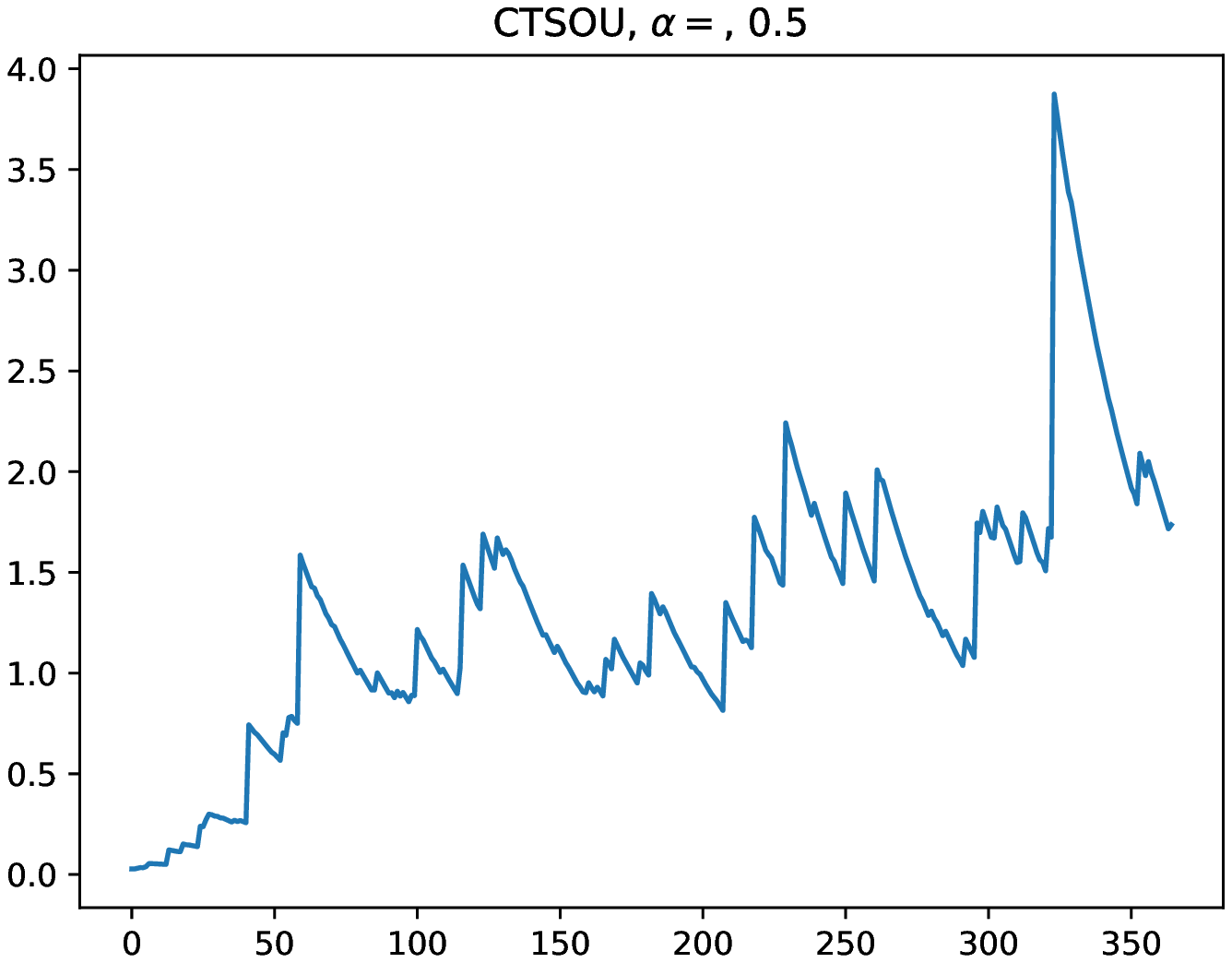}
                }
        \end{subfigure}
        \\
        \begin{subfigure}[c]{.5\textwidth}{
                \includegraphics[width=70mm]{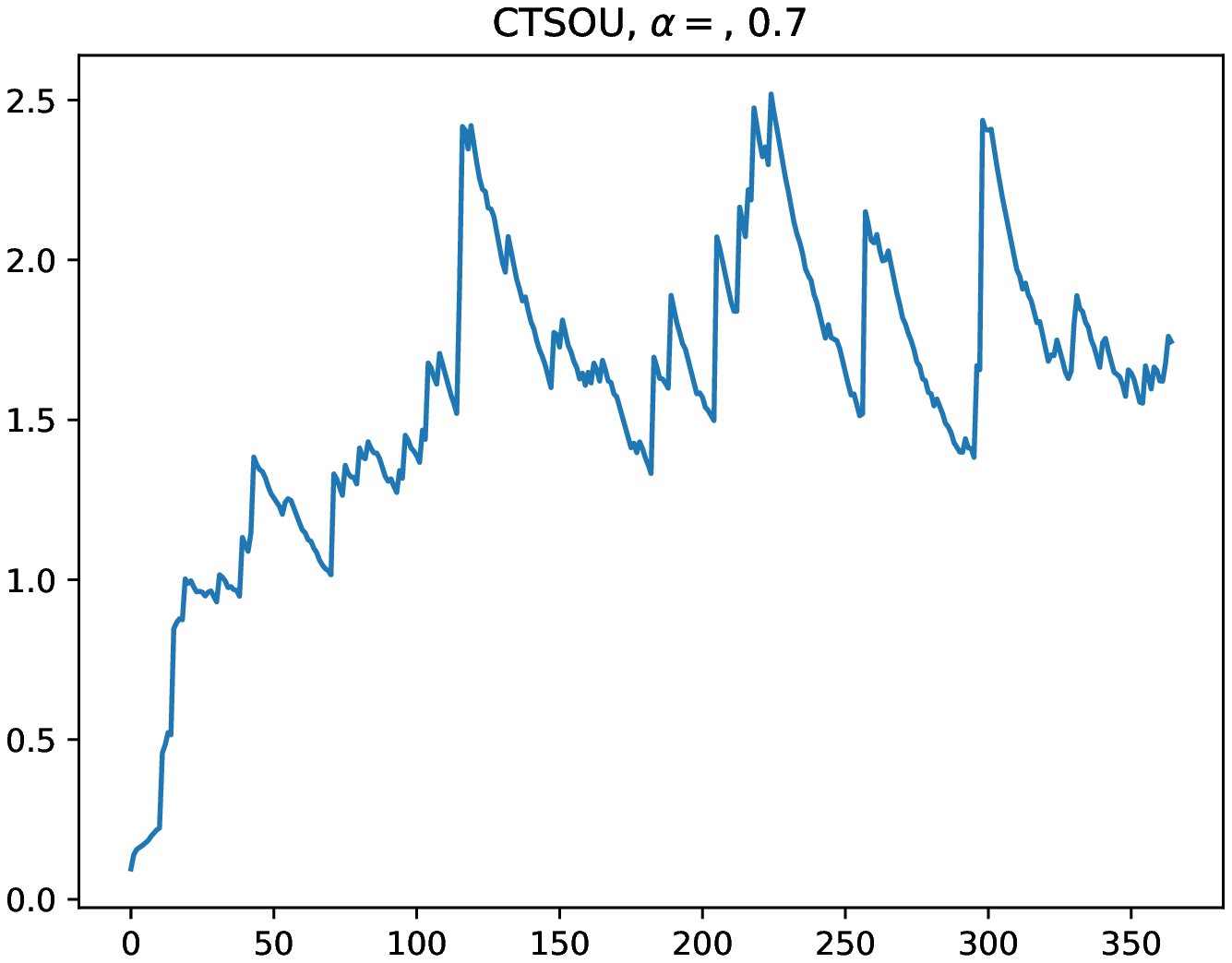}
                }
        \end{subfigure}
        \begin{subfigure}[c]{.5\textwidth}{
                \includegraphics[width=70mm]{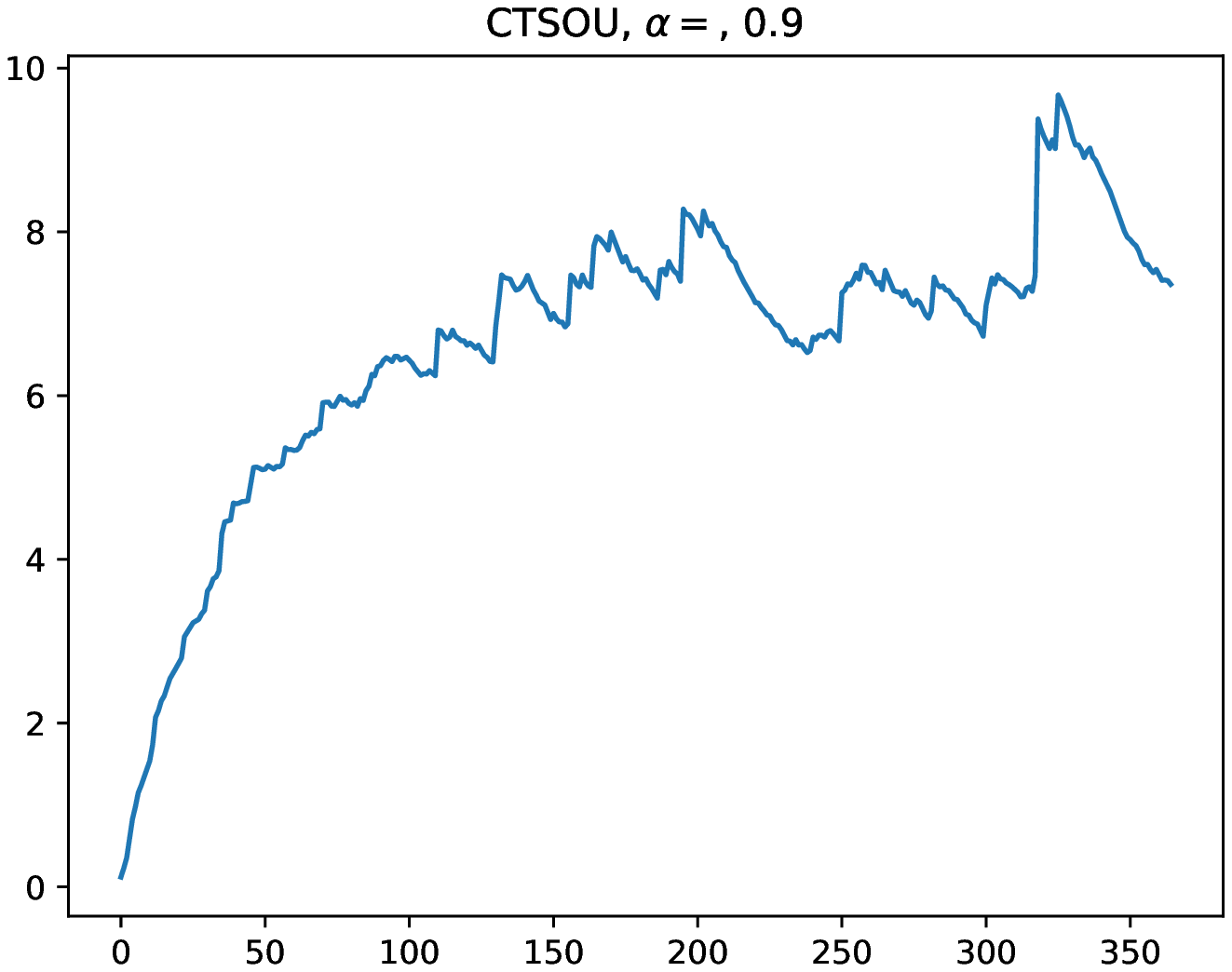}
                }
        \end{subfigure}
\end{figure}%
The Tables\myref{tab:cts:ou:1:365} and\myref{tab:cts:ou:30:365}
compare then the true values of the first four cumulants
$c_{X,k}(0,\Delta t)$ with their corresponding estimates from $10^6$
simulations respectively with $\Delta t=1/365$ and $\Delta
t=30/365$. We can conclude therefrom that the proposed
Algorithm\myref{alg:ctsou} produces unbiased cumulants that are very
close to their theoretical values. For the sake of brevity, we do
not report the additional results obtained with different parameter
settings that anyhow bring us to the same findings. Overall, from
the numerical results reported in this section, it is evident that
the Algorithm\myref{alg:ctsou} proposed above can achieve a very
high level of accuracy as well as a conspicuous efficiency.

The fact that we can easily compute the cumulants of an OU process
substantiates the advantages of focusing our treatment on the law of
the \arem\ of its stationary distribution. In addition to both the
simple derivation of the transition \pdf\ and the detailed testing
of its statistical properties, we could indeed also conceive a
parameter estimation procedure based on the generalized method of
moments (GMM). We remark finally that the law of an \arem\ always is
\id, and therefore a simple modification of the simulation procedure
presented in the Algorithm\myref{alg:ctsou} could be adopted for the
generation of a \Levy\ process whose law at time $T$ is that of the
\arem\ of a CTS distribution.

\input{./Tables/cts_ou_1_365.tex}
\input{./Tables/cts_ou_30_365.tex}


\subsection{OU-CTS processes}\label{sub:num:cts}
Here too we will benchmark the results of the numerical experiments
against the true values of the first four cumulants of OU-CTS
process at time $\Delta t$ with $X(0)=0$. From the
formula\refeqq{eq:cts:cumulants} for the cumulants of a
$\cts\left(\alpha, \beta, c\right)$ distribution,  and
from\refeqq{eq:cumulants:ou1} and\refeqq{eq:cumulants:ou2} we first
recover indeed the cumulants of $X(\Delta t)$ with the degenerate
initial condition $X_0=x_0$
\begin{equation}\label{eq:cumulants:oucts}
    c_{X,k}(x_0,\Delta t)=x_0e^{-b\Delta t}\delta_{k,1}+\frac{c\,(1-e^{-k\Delta t})}{Tb\,k\beta^{k-\alpha}}\,\Gamma\left(k-\alpha\right)
    \qquad\quad k=1,2,\ldots
\end{equation}
For our simulations we consider then the same parameter settings of
the previous section -- $\left(b, c, \beta\right) = \left(10, 0.8,
1.4\right)$ -- adapted to an OU-CTS process, and with
$\alpha\in\{0.3, 0.5, 0.7, 0.9\}$ we get the sample trajectories
displayed in the Figure\myref{fig:oucts:trajectories}, where of
course the case $\alpha=0.5$ is that of an OU-IG process.
\begin{figure}
\caption{Sample trajectories of OU-CTS processes with $\left(b, c,
\beta\right) = \left(10, 0.8, 1.4\right)$ and $\alpha\in\{0.3, 0.5,
0.7, 0.9\}$}\label{fig:oucts:trajectories}
        \begin{subfigure}[c]{.5\textwidth}{
                \includegraphics[width=70mm]{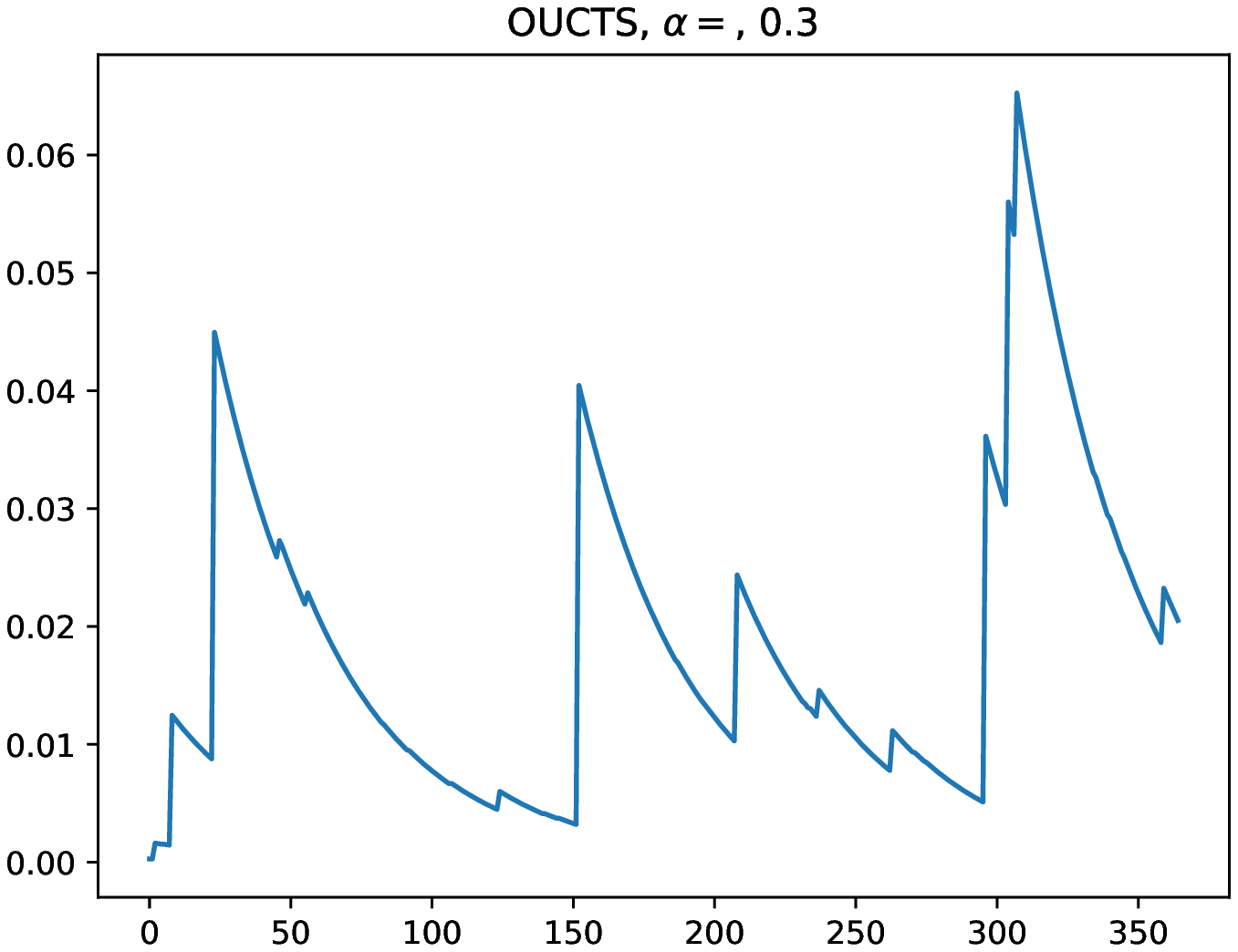}
                }
        \end{subfigure}
        \begin{subfigure}[c]{.5\textwidth}{
                \includegraphics[width=70mm]{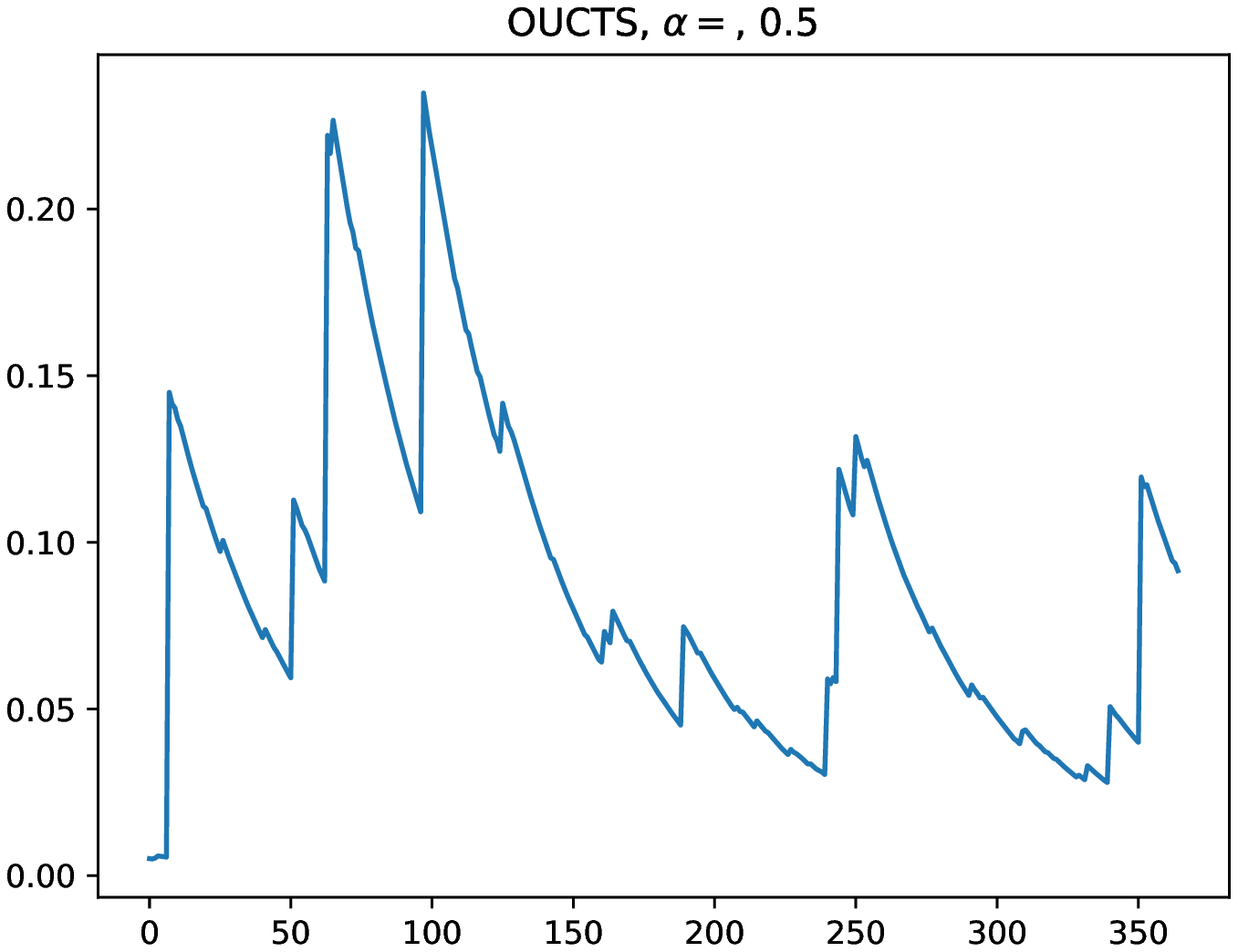}
                }
        \end{subfigure}
        \\
        \begin{subfigure}[c]{.5\textwidth}{
                \includegraphics[width=70mm]{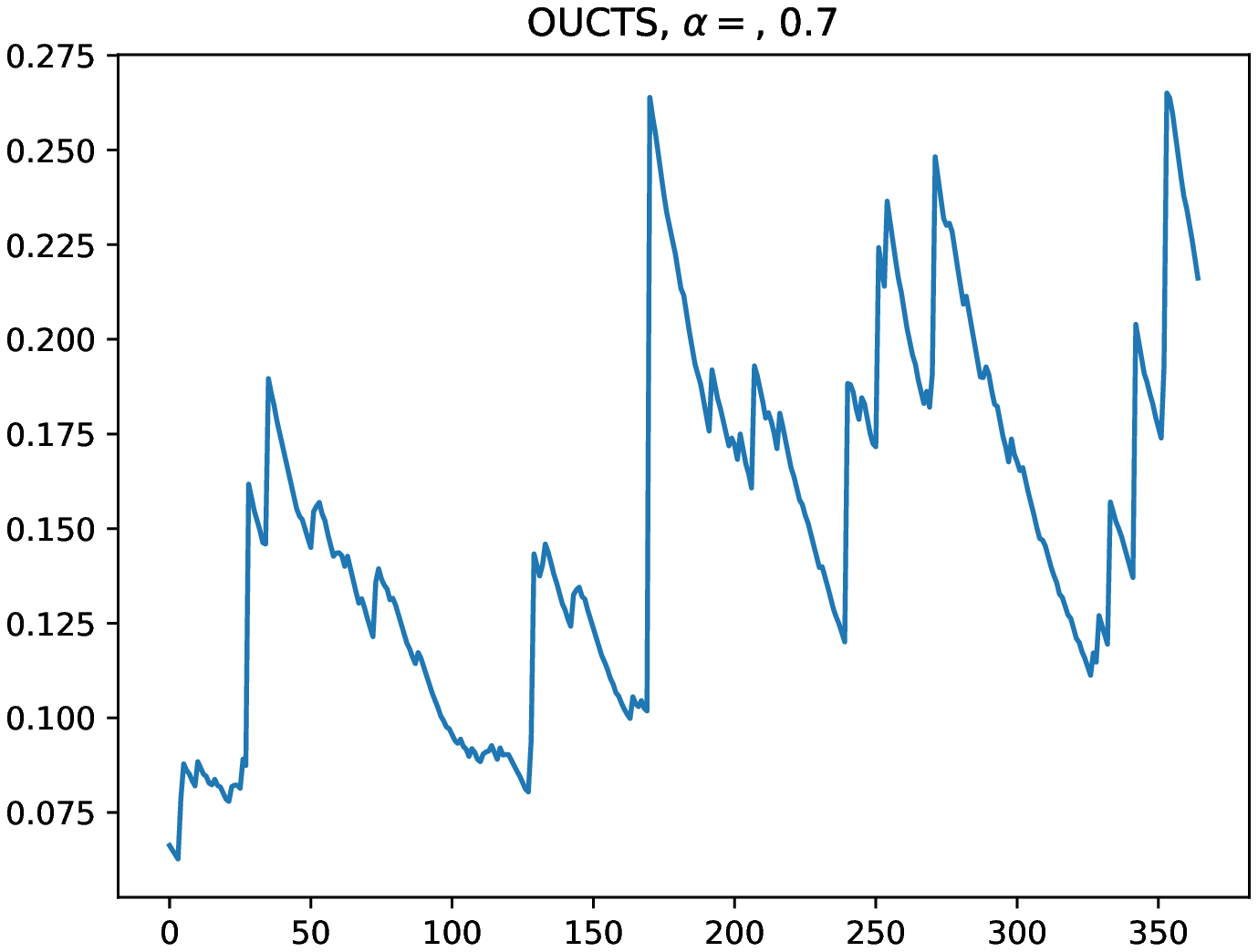}
                }
        \end{subfigure}
        \begin{subfigure}[c]{.5\textwidth}{
                \includegraphics[width=70mm]{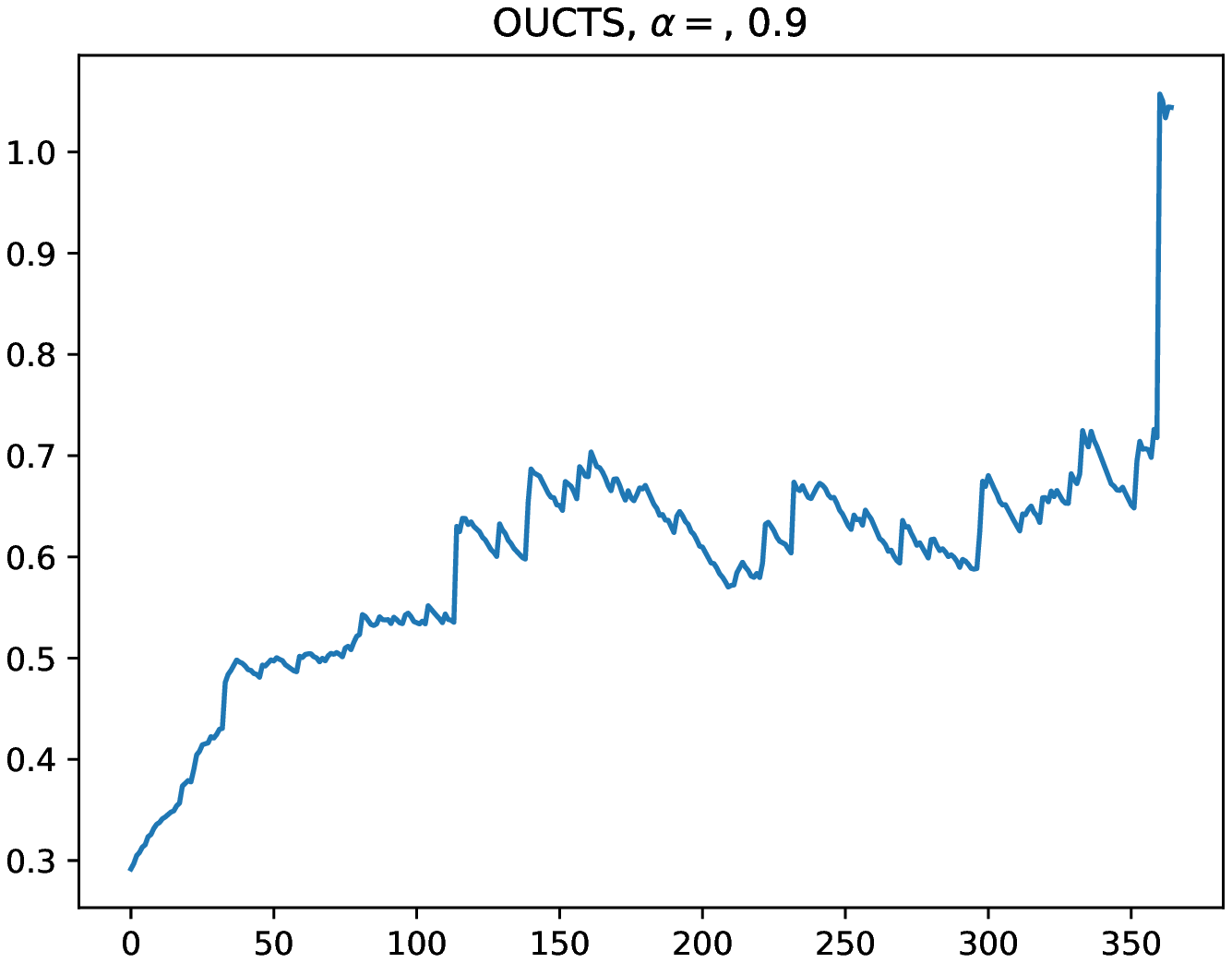}
                }
        \end{subfigure}
\end{figure}%

In addition to the Algorithm\myref{alg:ou}, to generate the OU-CTS
processes we will consider here two approximate procedures: the
first boils down to simply neglect $X_2$ in the
Proposition\myref{prop:ou:ts}; the second -- in the same vein of Benth
et al.\mycite{BDPL18} dealing with the \emph{normal inverse
Gaussian}-driven OU processes -- takes advantage of the
approximation of the law of $Z(t)$ in\refeqq{eq:sol:OU} with that of
$e^{-k\, t}L(t)$ where $L(t)\sim\cts\left(\alpha,\,^\beta/_a,\,
c\,t\right)$. The Tables\myref{tab:ou:cts:1:365}
and\myref{tab:ou:cts:30:365} compare then the true values of the
first four cumulants $c_{X,k}(0,\Delta t)$ with their corresponding
estimates for $10^6$ simulations with $\Delta t=1/365$ and $\Delta t=30/365$; the
labels \emph{$X_1$ only} and \emph{Approximation 2} refer to the
aforesaid first and second alternative procedures respectively.

From the Table\myref{tab:ou:cts:1:365} we can now conclude that our
\emph{Algorithm\myref{alg:ou}} has the lowest percent errors, but
nevertheless, in some practical situations, the errors of two other
approximations could be deemed acceptable taking also into account
that their computational cost is lower. In particular, the second
alternative procedure outperforms the third one and its percent
errors are not much higher than those of the exact method. When
however the time step is larger, or equivalently when $a=e^{-b\Delta t}$ is
close to $0$, the three procedures give radically different outcomes
and, as it is shown in the Table\myref{tab:ou:cts:30:365}, the two
approximate methods return completely biased results. Conversely,
the exact method continues to be reliable and its percent errors
remain small even for the higher cumulants.

The previous state of affairs for an OU-CTS is due to the fact that
$X_2$ in Proposition\myref{prop:ou:ts} produces only a second order
effect when $\Delta t\to0^+$: using indeed a Taylor expansion we
find
\begin{equation*}
    \Lambda_a=\frac{c\Gamma(1-\alpha)b\,\beta^{\alpha}}{2T}\,(\Delta
t)^2+o\big((\Delta t)^2\big)
\end{equation*}
and therefore the compound Poisson $X_2$ has a relevant impact only
when $\Delta t$ is not too small. Notice instead that this is not
how a CTS-OU process behaves because from the
Proposition\myref{prop:ctsou} we see that for $\Delta t\to0^+$
\begin{equation*}
   \Lambda_a=c\Gamma(1-\alpha)b\,\beta^{\alpha}\Delta t+o(\Delta
   t)
\end{equation*}
so that $X_2$ results in a first order effect and cannot be
neglected even for small $\Delta t$. As mentioned above, to tackle
the simulation of the random rate, Qu et al.\mycite{QDZ20} have
proposed an alternative solution that is based on an acceptance
rejection method again, and whose expected number of iterations
before acceptance tends to $1$ for small time steps ($\Delta
t\to0^+;\;a\rightarrow 1^-$); but unfortunately this value somehow
deteriorates and tends to $2$ ($50\%$ of acceptance) for large time
steps ($\Delta t\to+\infty;\;a\rightarrow 0^+$). From our previous
findings we know instead that neglecting $X_2$ is a fair
approximation for finer time grids so that the impact on the
computational cost of the acceptance rejection is rather restricted.
On the other hand, no matter how large the time step $\Delta t$ is,
with our algorithm the expected number of iterations before
acceptance can be kept as close to $1$ as possible because it
depends on the accuracy of a trapezioidal approximation. Therefore,
recalling also that the computational cost to generate a simple
discrete \rv\ is very low, our approach turns out to be
computationally more efficient.

These observations could lead to a convenient strategy combining
parameters estimation and exact simulation of the OU-CTS processes.
Assuming that the data could be made available with a fine enough
time-granularity (e.g. daily $t=1/365$), we could base the
parameters estimation on the likelihood methods by approximating the
exact transition \pdf\ with that of a CTS law $\cts\left(\alpha,
\frac{\beta}{a}, \frac{c\left(1 -
a^{\alpha}\right)}{T\alpha\,b}\right)$. However, to avoid being
forced to always simulate the OU-CTS processes on a fine time-grid
allowing the approximations (for instance if one needs to simulate
it at a monthly granularity $t=30/365$), the generation of the
skeleton of such processes will be preferably based on the exact
method of the Algorthm\myref{alg:ou}.

\input{./Tables/ou_cts_1_365.tex}
\input{./Tables/ou_cts_30_365.tex}

\section{Conclusions}\label{sec:conclusions}
In this paper we have studied the transition laws of the tempered
stable related OU processes with finite variation from the
standpoint of the \arem s of \sd\ distributions: in fact, the
transition law of any OU process essentially coincides with the
distribution of the \arem\ of its stationary \sd\ distribution. To
this purpose, we first derived the \Levy\ triplet of the \arem\ of a
general \sd\ law that is then instrumental to find the
representation of the transition law of tempered stable related OU
processes with finite variation. We thereafter focused our attention
on the CTS-OU and the OU-CTS processes: respectively those whose
stationary law is a CTS distribution, and those whose BDLP is a CTS
process. As already done in Zhang\mycite{ZZhang07}, Kawai and
Masuda\mycite{KawaiMasuda_1} and Qu et al.\mycite{QDZ20}, we have
shown that their transition law coincides with the distribution of
the sum of a CTS distributed \rv\, (with scaled parameters), of a
suitable compound Poisson \rv\ and of a degenerate term: we
accordingly also derived their path-generation algorithms.

As for the simulation of the skeleton of CTS-OU processes, our
proposed procedure amounts to an improvement with respect to the
existing solutions presented in Zhang and Zhang\mycite{ZZhang07},
Zhang\mycite{Zhang11}, Kawai and Masuda\mycite{KawaiMasuda_1}:
indeed it does not rely on additional acceptance rejection methods
other than that required to generate a CTS distributed \rv. On the
other hand, also the simulation procedure for a OU-CTS process is
based on an acceptance rejection approach more efficient than that
described in Qu et al.\mycite{QDZ20}, because here the number of
iterations before acceptance can be made arbitrarily close to $1$ no
matter how fine we choose the time grid of the skeleton.

Although we have considered in the present paper only the CTS
distributions restricted on the positive real axis, the results can
be easily extended to the bilateral case and the simulation of the
relative processes would be simply obtained by running twice the
proposed algorithms. A further object of our future inquiries will
be instead the possible extension to the $p\,$-TS related OU
processes combined with the application of the algorithms recently
proposed in Grabchak\mycite{Grabchak20b} to draw samples from
$p\,$-TS laws. We remark moreover that, due to the fact that the
laws of the \arem s are \id, our approach is also suited to build
and simulate new \Levy\ processes via the subordination of a Brownian
motion with the \Levy\ process generated by the \arem\ of a gamma
and IG law, respectively (as done for instance in Gardini
et al.\mycite{Gardini20b, Gardini20a}).

All these algorithms would finally be especially useful for a
simulation-based statistical inference, and for some financial
applications like as the derivative pricing and the value-risk
calculations. To this end, a possible future research line could be
the study of the time reversal simulations in the spirit of some
recent papers by Pellegrino and Sabino\mycite{PellegrinoSabino15}
and Sabino\mycite{Sabino20a} relatively to the time-changed OU
processes introduced in Li and Linetsky\mycite{LiLinetsky2014}.

        \bibliographystyle{plain}
        \bibliography{Cufaro_Sabino_CTSOUCTS}
\end{document}

%% file: Tables/cts_ou_1_365.tex
\begin{table}[ht!]
    \centering\scriptsize
        \resizebox{\textwidth}{!}{
        \begin{tabular}{*{13}{|c|rrr|rrr|rrr|rrr}}
                    \hline
&       \multicolumn{3}{c|}{$c_{X, 1}(0,\Delta t)$} & \multicolumn{3}{c|}{$c_{X, 2}(0,\Delta t)$} & \multicolumn{3}{c|}{$c_{X, 3}(0,\Delta t)$} & \multicolumn{3}{c|}{$c_{X, 4}(0,\Delta t)$} \\
                        \hline
                                \multicolumn{13}{|c|}{Algorithm\myref{alg:ctsou}}\\
                    \hline
                    $\alpha$ & true & MC & err \% & true & MC & err \% & true & MC & err \% & true & MC & err \% \\
                    \hline
$0.1$ & $1.71$ & $1.72$ & $-0.6$ & $2.16$ & $2.18$ & $-0.9$ & $4.35$ & $4.39$ & $-0.9$ & $11.85$ & $11.91$ & $-0.6$\\
$0.3$ & $2.22$ & $2.22$ & $0.0$ & $2.19$ & $2.20$ & $-0.7$ & $3.93$ & $4.00$ & $-1.8$ & $9.97$ & $10.25$ & $-2.8$\\
$0.5$ & $3.24$ & $3.24$ & $0.1$ & $2.28$ & $2.28$ & $0.3$ & $3.62$ & $3.63$ & $-0.4$ & $8.50$ & $8.63$ & $-1.6$\\
$0.7$ & $5.85$ & $5.85$ & $-0.1$ & $2.47$ & $2.48$ & $-0.4$ & $3.40$ & $3.42$ & $-0.7$ & $7.34$ & $7.44$ & $-1.3$\\
$0.9$ & $19.89$ & $19.90$ & $-0.1$ & $2.80$ & $2.82$ & $-0.7$ & $3.26$ & $3.29$ & $-0.9$ & $6.43$ & $6.40$ & $0.4$\\
                    \hline
        \end{tabular}
        }
    \scriptsize
    \caption{\footnotesize{Comparing the first four true cumulants
    with their corresponding MC-estimated values (multiplied by $100$)
obtained with $10^6$ simulations and $\Delta t=1/365$.}}\label{tab:cts:ou:1:365}
\end{table}

%% file: Tables/cts_ou_30_365.tex
\begin{table}[ht!]
    \centering\scriptsize
        \resizebox{\textwidth}{!}{
        \begin{tabular}{*{13}{|c|rrr|rrr|rrr|rrr}}
                    \hline
&       \multicolumn{3}{c|}{$c_{X, 1}(0,\Delta t)$} & \multicolumn{3}{c|}{$c_{X, 2}(0,\Delta t)$} & \multicolumn{3}{c|}{$c_{X, 3}(0,\Delta t)$} & \multicolumn{3}{c|}{$c_{X, 4}(0,\Delta t)$} \\
                        \hline
                                \multicolumn{13}{|c|}{Algorithm\myref{alg:ctsou}}\\
                    \hline
                    $\alpha$ & true & MC & err \% & true & MC & err \% & true & MC & err \% & true & MC & err \% \\
                    \hline
$0.1$ & $3.54$ & $3.60$ & $-1.8$ & $3.28$ & $3.33$ & $-1.6$ & $5.04$ & $5.16$ & $-2.4$ & $10.99$ & $11.57$ & $-5.3$\\
$0.3$ & $4.60$ & $4.65$ & $-1.1$ & $3.31$ & $3.40$ & $-2.8$ & $4.56$ & $4.73$ & $-3.8$ & $9.25$ & $9.77$ & $-5.7$\\
$0.5$ & $6.72$ & $6.81$ & $-1.4$ & $3.45$ & $3.50$ & $-1.5$ & $4.20$ & $4.40$ & $-5.0$ & $7.88$ & $8.15$ & $-3.4$\\
$0.7$ & $12.12$ & $12.20$ & $-0.6$ & $3.74$ & $3.82$ & $-2.1$ & $3.94$ & $4.05$ & $-2.8$ & $6.81$ & $6.98$ & $-2.5$\\
$0.9$ & $41.24$ & $41.28$ & $-0.1$ & $4.24$ & $4.27$ & $-0.7$ & $3.78$ & $3.82$ & $-1.2$ & $5.96$ & $6.06$ & $-1.6$\\
                        \hline
        \end{tabular}
        }
    \scriptsize
    \caption{\footnotesize{Comparing the first four true cumulants
    with their corresponding MC-estimated values (multiplied by $10$)
obtained with $10^6$ simulations and $\Delta t=30/365$.}}\label{tab:cts:ou:30:365}
\end{table}

%% file: Tables/ou_cts_1_365.tex
\begin{table}[ht!]
    \centering\scriptsize
        \resizebox{\textwidth}{!}{
        \begin{tabular}{*{13}{|c|rrr|rrr|rrr|rrr}}
                    \hline
&       \multicolumn{3}{c|}{$c_{X, 1}(0,\Delta t)$} & \multicolumn{3}{c|}{$c_{X, 2}(0,\Delta t)$} & \multicolumn{3}{c|}{$c_{X, 3}(0,\Delta t)$} & \multicolumn{3}{c|}{$c_{X, 4}(0,\Delta t)$} \\
                        \hline
                                \multicolumn{13}{|c|}{Algorithm\myref{alg:ou}}\\
                    \hline
                    $\alpha$ & true & MC & err \% & true & MC & err \% & true & MC & err \% & true & MC & err \% \\
                    \hline
$0.1$ & $1.71$ & $1.70$ & $0.2$ & $1.08$ & $1.07$ & $0.7$ & $1.45$ & $1.43$ & $1.0$ & $2.96$ & $2.91$ & $1.7$\\
$0.3$ & $2.22$ & $2.22$ & $-0.1$ & $1.09$ & $1.08$ & $0.9$ & $1.31$ & $1.27$ & $3.2$ & $2.49$ & $2.43$ & $2.5$\\
$0.5$ & $3.24$ & $3.23$ & $0.2$ & $1.14$ & $1.15$ & $-0.5$ & $1.21$ & $1.25$ & $-3.6$ & $2.12$ & $2.18$ & $-2.7$\\
$0.7$ & $5.85$ & $5.85$ & $0.0$ & $1.24$ & $1.24$ & $-0.4$ & $1.13$ & $1.16$ & $-2.4$ & $1.84$ & $1.91$ & $-4.2$\\
$0.9$ & $19.89$ & $19.88$ & $0.0$ & $1.40$ & $1.39$ & $0.6$ & $1.09$ & $1.06$ & $2.4$ & $1.61$ & $1.54$ & $4.2$\\
                        \hline
                    \multicolumn{13}{|c|}{$X_1$ only }\\
                    \hline
                    $\alpha$ & true & MC & err \% & true & MC & err \% & true & MC & err \% & true & MC & err \% \\
                    \hline
$0.1$ & $1.71$ & $1.68$ & $1.8$ & $1.08$ & $1.04$ & $4.2$ & $1.45$ & $1.35$ & $6.7$ & $2.96$ & $2.72$ & $8.2$\\
$0.3$ & $2.22$ & $2.19$ & $1.4$ & $1.09$ & $1.06$ & $3.3$ & $1.31$ & $1.24$ & $5.6$ & $2.49$ & $2.28$ & $8.6$\\
$0.5$ & $3.24$ & $3.22$ & $0.6$ & $1.14$ & $1.11$ & $2.5$ & $1.21$ & $1.14$ & $5.2$ & $2.12$ & $1.94$ & $8.9$\\
$0.7$ & $5.85$ & $5.82$ & $0.4$ & $1.24$ & $1.21$ & $2.1$ & $1.13$ & $1.09$ & $4.1$ & $1.84$ & $1.73$ & $5.8$\\
$0.9$ & $19.89$ & $19.86$ & $0.1$ & $1.40$ & $1.37$ & $2.2$ & $1.09$ & $1.01$ & $6.9$ & $1.61$ & $1.36$ & $15.6$\\
                    \hline
                        \multicolumn{13}{|c|}{Approximation 2 }\\
                    \hline
                    $\alpha$ & true & MC & err \% & true & MC & err \% & true & MC & err \% & true & MC & err \% \\
                    \hline
$0.1$ & $1.71$ & $1.63$ & $4.3$ & $1.08$ & $1.06$ & $2.1$ & $1.45$ & $1.56$ & $-7.5$ & $2.96$ & $3.62$ & $-22.4$\\
$0.3$ & $2.22$ & $2.16$ & $2.4$ & $1.09$ & $1.00$ & $8.6$ & $1.31$ & $1.05$ & $19.6$ & $2.49$ & $1.63$ & $34.7$\\
$0.5$ & $3.24$ & $3.22$ & $0.5$ & $1.14$ & $1.17$ & $-2.5$ & $1.21$ & $1.34$ & $-11.0$ & $2.12$ & $2.54$ & $-19.7$\\
$0.7$ & $5.85$ & $5.68$ & $2.8$ & $1.24$ & $1.16$ & $6.1$ & $1.13$ & $0.95$ & $16.3$ & $1.84$ & $1.37$ & $25.4$\\
$0.9$ & $19.89$ & $19.58$ & $1.6$ & $1.40$ & $1.35$ & $3.8$ & $1.09$ & $1.04$ & $3.8$ & $1.61$ & $1.41$ & $12.5$\\

                    \hline
        \end{tabular}
        }
    \scriptsize
    \caption{\footnotesize{Comparing the first four true cumulants
    with their corresponding MC-estimated values (multiplied by $1000$)
obtained with $10^6$ simulations and $\Delta t=1/365$. The exact method
uses subdivisions in $L=10$ intervals.}}\label{tab:ou:cts:1:365}
\end{table}

%% file: Tables/ou_cts_30_365.tex
\begin{table}[ht!]
    \centering\scriptsize
        \resizebox{\textwidth}{!}{
        \begin{tabular}{*{13}{|c|rrr|rrr|rrr|rrr}}
                    \hline
&       \multicolumn{3}{c|}{$c_{X, 1}(0,\Delta t)$} & \multicolumn{3}{c|}{$c_{X, 2}(0,\Delta t)$} & \multicolumn{3}{c|}{$c_{X, 3}(0,\Delta t)$} & \multicolumn{3}{c|}{$c_{X, 4}(0,\Delta t)$} \\
                        \hline
                                \multicolumn{13}{|c|}{Algorithm\myref{alg:ou}}\\
                    \hline
                    $\alpha$ & true & MC & err \% & true & MC & err \% & true & MC & err \% & true & MC & err \% \\
                    \hline
$0.1$ & $3.54$ & $3.54$ & $0.1$ & $1.64$ & $1.65$ & $-0.8$ & $1.68$ & $1.73$ & $-2.7$ & $2.75$ & $2.89$ & $-5.2$\\
$0.3$ & $4.60$ & $4.58$ & $0.4$ & $1.65$ & $1.62$ & $1.9$ & $1.52$ & $1.49$ & $2.0$ & $2.31$ & $2.36$ & $-2.1$\\
$0.5$ & $6.72$ & $6.72$ & $-0.1$ & $1.73$ & $1.74$ & $-1.0$ & $1.40$ & $1.43$ & $-2.0$ & $1.97$ & $1.98$ & $-0.3$\\
$0.7$ & $12.12$ & $12.15$ & $-0.2$ & $1.87$ & $1.88$ & $-0.4$ & $1.31$ & $1.31$ & $0.6$ & $1.70$ & $1.63$ & $4.0$\\
$0.9$ & $41.24$ & $41.22$ & $0.1$ & $2.12$ & $2.12$ & $0.2$ & $1.26$ & $1.25$ & $0.8$ & $1.49$ & $1.42$ & $4.6$\\
                        \hline
                    \multicolumn{13}{|c|}{$X_1$ only}\\
                    \hline
                    $\alpha$ & true & MC & err \ & true & MC & err \ & true & MC & err \ & true & MC & err \ \\
                    \hline
$0.1$ & $3.54$ & $2.37$ & $33.0$ & $1.64$ & $0.67$ & $58.9$ & $1.68$ & $0.41$ & $75.7$ & $2.75$ & $0.39$ & $85.8$\\
$0.3$ & $4.60$ & $3.35$ & $27.1$ & $1.65$ & $0.74$ & $55.6$ & $1.52$ & $0.39$ & $74.2$ & $2.31$ & $0.33$ & $85.6$\\
$0.5$ & $6.72$ & $5.34$ & $20.4$ & $1.73$ & $0.84$ & $51.4$ & $1.40$ & $0.40$ & $71.4$ & $1.97$ & $0.33$ & $83.3$\\
$0.7$ & $12.12$ & $10.57$ & $12.8$ & $1.87$ & $1.00$ & $46.8$ & $1.31$ & $0.40$ & $69.3$ & $1.70$ & $0.29$ & $83.2$\\
$0.9$ & $41.24$ & $39.37$ & $4.5$ & $2.12$ & $1.24$ & $41.6$ & $1.26$ & $0.44$ & $65.2$ & $1.49$ & $0.32$ & $78.3$\\
                    \hline
                        \multicolumn{13}{|c|}{Approximation 2 }\\
                    \hline
                    $\alpha$ & true & MC & err \ & true & MC & err \ & true & MC & err \ & true & MC & err \ \\
                    \hline
$0.1$ & $3.54$ & $2.29$ & $35.4$ & $1.64$ & $0.65$ & $60.4$ & $1.68$ & $0.39$ & $77.0$ & $2.75$ & $0.35$ & $87.3$\\
$0.3$ & $4.60$ & $2.95$ & $35.8$ & $1.65$ & $0.65$ & $60.8$ & $1.52$ & $0.34$ & $77.5$ & $2.31$ & $0.29$ & $87.7$\\
$0.5$ & $6.72$ & $4.34$ & $35.4$ & $1.73$ & $0.69$ & $60.2$ & $1.40$ & $0.33$ & $76.7$ & $1.97$ & $0.26$ & $86.9$\\
$0.7$ & $12.12$ & $7.82$ & $35.5$ & $1.87$ & $0.74$ & $60.3$ & $1.31$ & $0.31$ & $76.6$ & $1.70$ & $0.23$ & $86.5$\\
$0.9$ & $41.24$ & $26.60$ & $35.5$ & $2.12$ & $0.84$ & $60.5$ & $1.26$ & $0.29$ & $76.9$ & $1.49$ & $0.20$ & $86.8$\\
                    \hline
        \end{tabular}
        }
    \scriptsize
    \caption{\footnotesize{Comparing the first four true cumulants
    with their corresponding MC-estimated values (multiplied by $100$)
obtained with $10^6$ simulations and $\Delta t=30/365$. The exact method
uses subdivisions in $L=10$ intervals.}}\label{tab:ou:cts:30:365}
\end{table}